\documentclass[lettersize,journal]{IEEEtran}
\usepackage{amsmath,amsfonts}
\usepackage{algorithmic}
\usepackage{algorithm}
\usepackage{array}
\usepackage[caption=false,font=footnotesize,labelfont=sf,textfont=sf]{subfig}
\usepackage{textcomp}
\usepackage{url}
\usepackage{verbatim}
\usepackage{graphicx}
\usepackage{cite}
\usepackage[dvips]{epsfig}
\usepackage{latexsym,amssymb}
\usepackage{amsmath, amsbsy, amsthm}
\usepackage{amsopn, amstext}
\usepackage{colortbl}
\usepackage{microtype}
\usepackage{color}
\usepackage{bm}
\usepackage{booktabs}
\usepackage{comment}
\usepackage{booktabs} 
\usepackage[hidelinks]{hyperref}
\newcommand{\be}{\begin{eqnarray}}
\newcommand{\ee}{\end{eqnarray}}

\newtheorem{problem}{Problem}
\newtheorem{theorem}{Theorem}

\newtheorem{remark}{Remark}
\newtheorem{definition}{Definition}
\newtheorem{assumption}{Assumption}
\newtheorem{lemma}{Lemma}

\renewcommand{\zeta}{\Gamma}

\allowdisplaybreaks
\begin{document}

\title{\Large Decentralized Estimation and Control for Leader-Follower Networked Systems with Asymmetric Information Structure}

\author{Yiting Luo, Wei Wang and Qingyuan Qi*
\thanks{This work was supported by National Natural Science Foundation of China under grants U24A20281, 61903210, Qingdao Natural Science Foundation Original Exploration Project under grant 24-4-zrjj-130-jch, Natural Science Foundation of Heilongjiang Province under grant LH2023F021, and Shandong Province Youth Innovation Team Plan under grant 2023KJ083. \textit{(Corresponding author: Qingyuan Qi.)}}
\thanks{Yiting Luo and Qingyuan Qi are with College of Intelligent Systems Science and Engineering and Institute of Qingdao Innovation and Development Center, Harbin Engineering University, Harbin 150001, China (e-mail: yiting@hrbeu.edu.cn; qingyuan.qi@hrbeu.edu.cn).}
\thanks{Wei Wang is with School of Control Science and Engineering, Shandong University, Jinan 250061, China (e-mail: w.wang@sdu.edu.cn).}
}        

\maketitle
                           
\begin{abstract} 
  In this paper, the decentralized estimation and linear quadratic (LQ) control problem for a leader-follower networked system (LFNS) is studied from the perspective of asymmetric information. Specifically, for a leader-follower network, the follower agent will be affected by the leader agent, while the follower agent will not affect the leader agent. Hence, the information sets accessed by the control variables of the leader agent and the follower agent are asymmetric, which will bring essential difficulties in finding the optimal control strategy. To this end, the orthogonal decomposition method is adopted to achieve the main results. The main contributions of this paper can be summarized as follows: Firstly, the optimal iterative estimation is derived using the conditional independence property established in this paper. Secondly, the optimal decentralized control strategy is derived by decoupling the forward-backward stochastic difference equations (FBSDEs), based on the derived optimal iterative estimation. Thirdly, the necessary and sufficient conditions for the feedback stabilization of the LFNS in infinite-horizon are derived. Finally, the proposed theoretical results are applied to solve the decentralized control problem of a leader-follower autonomous underwater vehicle (LF-AUV) system. The optimal control inputs for the AUVs are provided, and simulation results verify the effectiveness of the obtained results.                        
\end{abstract}


\def\s{\small}
\def\n{\normalsize}
\def\+{\!+\!}
\def\-{\!-\!}
\def\={\!=\!}
\def\E{\mathbb{E}}
\def\P{\mathbb{P}}
\def\Tr{\mathrm{Tr}}
\def\b{\color{blue}}
\def\r{\color{red}}

\section{Introduction}

Decentralized optimal control, as a specialized class of optimal control problems, has attracted considerable research attention in recent years, yielding significant theoretical and practical advancements \cite{yx2010,lm2011,ml2014,rrj2016,qxz2021}. Unlike centralized frameworks where a single agent possesses full system information, decentralized control inherently involves multiple agents that independently optimize a global performance metric using their local information sets. This distributed structure introduces unique challenges, particularly due to inter-agent coupling, which often complicates the derivation of optimal control strategies. A canonical example is the Witsenhausen's counterexample, which was demonstrated in \cite{w1968} that even for linear systems with asymmetric information, the optimal control strategy may be nonlinear, and an explicit analytical solution for this strategy has yet to be derived. Building on the seminal work of \cite{w1968}, research on optimal decentralized control has found widespread applications across diverse fields, including game-theoretic control, optimal local and remote control, and networked control. Firstly, in game-theoretic control, the interplay between strategic decision-making and information asymmetry has been studied. For example, tractable solutions for finite-horizon games and LQ settings were proposed by decoupling shared and private information structure \cite{nglb2014}. This approach was further extended to a dynamic multiplayer nonzero sum game with asymmetric information, as shown in \cite{glb2017}. Subsequently, in optimal local and remote control, global optimality of systems with asymmetric information was achieved through the coordination between the local controller and the remote controller. For instance, a dynamic program of the decentralized optimal control problem for a linear plant controlled by two controllers was provided by using the common information approach \cite{oan2016}. Building on this, the decentralized optimal control problem of systems with remote and local controllers was solved in \cite{lx2018}, where stabilization of the system was further considered. Recently, the jointly decentralized optimal local and remote controls for multiple systems were investigated via decoupling a group of FBSDEs \cite{qxzl2025}. Finally, in networked control, the decentralized controls for networked control systems with multiple controllers of asymmetric information were investigated in \cite{lqzx2021, wlll2024}. Specifically, the linear optimal decentralized control for networked control systems with asymmetric information was derived in \cite{lqzx2021}. As an extension of \cite{lqzx2021}, the closed-loop decentralized control was investigated for interconnected networked systems with asymmetric information, the optimal control strategy was derived, and stabilization conditions were proposed, \cite{wlll2024}.

This paper focuses on the decentralized control problem for a class of LFNSs, as illustrated in Fig. 1.
\begin{figure}
\centerline{\includegraphics[width=3.2in]{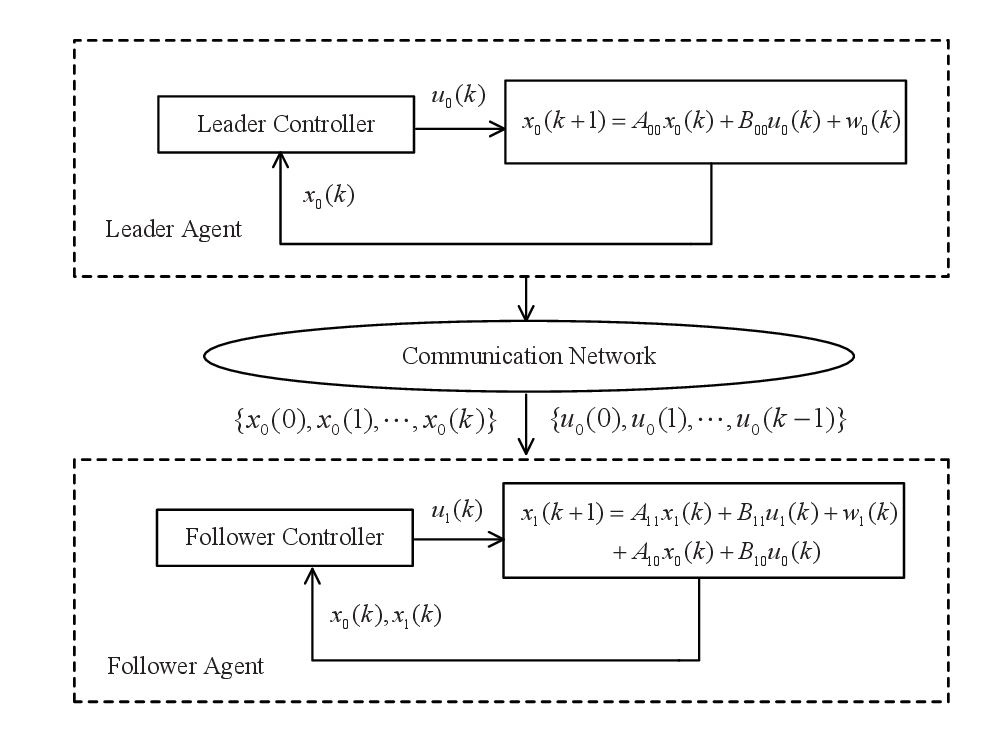}}
\caption{Schematic diagram of leader-follower networked system (LFNS).}
\label{6}
\end{figure}
In such systems, the states and actions of the leader agent influence the dynamics of the follower agent, while the follower's states and actions have no reciprocal effect on the leader agent. This unidirectional interaction creates a fundamental asymmetry in information availability: the leader agent has access to its own state and control history, whereas the follower agent operates under a nested information structure, relying on its local observations and the leader agent's influence. Specifically, the leader agent's information set, denoted as $\mathcal{S}_0(k)$, includes its state trajectories and control inputs, while the follower agent's information set, $\mathcal{S}_1(k)$, comprises its local state measurements and the leader agent's observable actions. This disparity in information access introduces significant challenges in deriving optimal control strategies, as the follower agent must infer the leader agent's intentions indirectly, while the leader agent must account for the follower agent's response in its decision-making. Despite these complexities, both agents share a common objective: minimizing a global cost function that reflects the collective performance of the system. This cost function typically penalizes deviations from desired states and excessive control efforts, balancing system performance with energy efficiency. By addressing these challenges, this paper aims to provide a comprehensive study for decentralized estimation and control of LFNSs with asymmetric information structure.

 The optimal decentralized control problem for LFNSs was first studied in \cite{hc1972}, where a partially nested information structure was defined and an optimal linear quadratic Gaussian (LQG) controller was derived. Afterwards, the optimal controller subject to a decentralization state feedback constraint based on a novel dynamic programming method was constructed in \cite{sl2010}. Subsequently, these results were extended by deriving explicit optimal controllers for the output feedback case of LQG systems with partially nested information structure under asymmetric information-sharing patterns in \cite{nkj2018}. Recently, the decentralized LQ control was investigated for non-Gaussian systems with major and minor agents in \cite{am2022}, where the conditional independence of the states of the minor agents was established. It is worth noting that while the considered decentralized control problem for LFNSs has been explored in prior works \cite{hc1972,sl2010,nkj2018,am2022}, a comprehensive theoretical solution remains elusive. The reasons are as follows: (1) the nested and asymmetric information structure complicates the derivation of globally optimal solutions, as the follower agent’s decisions depend on the leader agent’s actions, which are not directly observable; (2) the coupling between forward and backward stochastic dynamics introduces nonlinearities that are difficult to decouple without advanced mathematical tools; and (3) the lack of a unified framework to simultaneously address estimation and control under asymmetric information constraints. These challenges underscore the need for novel methodologies to bridge the gap between theoretical insights and practical solutions, which motivates the contributions of this paper.

Beyond its theoretical implications, the decentralized control problem for LFNSs holds potential practical value in emerging applications \cite{ils2007,vrc2011,dzry2016}. With the rapid advancement of networked autonomous systems, particularly in swarm robotics and unmanned systems, many real-world challenges can be framed as decentralized optimal control problems. For example, in AUV formation tracking, a leader AUV typically guides follower AUVs to collaboratively track a desired trajectory while avoiding collisions and environmental disturbances. However, existing studies \cite{gmt2007,lyz2019,wpz2023,dzlcwm2024,lnhkb2024} often simplify the problem by assuming symmetric information structure or neglecting the impact of asymmetric interactions. For instance, as demonstrated in \cite{gmt2007}, prior approaches predominantly focus on stability guarantees under idealized communication conditions, overlooking the inherent information asymmetry between leader and follower agents. This simplification limits their applicability to real-world scenarios where follower agents lack direct access to the leader agent's full states or decision-making logic. In practice, asymmetric information structure arises naturally due to sensor limitations, communication latency, or hierarchical operational roles, as observed in recent field experiments with multi-AUVs systems \cite{yxyg2023}. 

This paper proposes a systematic approach to decentralized estimation and control for LFNSs under asymmetric information. By integrating orthogonal decomposition with FBSDEs, we address the inherent challenges of nested information structure and inter-agent coupling. Overall, the innovative contributions of this paper lie primarily in the following aspects: By developing the conditional independence property, an optimal iterative estimator for the LFNS with asymmetric information structure is proposed. Based on this estimator, the optimal decentralized control strategy is derived by decoupling the FBSDEs, the challenge posed by asymmetric information is addressed through the application of an orthogonal decomposition method. Furthermore, the necessary and sufficient conditions for feedback stabilization of the LFNS in infinite-horizon scenarios are characterized. To validate the theoretical results, the proposed framework is successfully applied to a LF-AUV system, with numerical simulations demonstrating its practical effectiveness. 

The remainder of this paper is given as follows. Section \ref{ss2} discusses the existence and uniqueness of the optimal decentralized control strategy for the LFNS. In Section \ref{ss3}, the optimal iterative estimation is derived. In Section \ref{ss4}, the optimal decentralized control strategy of the LFNS in finite-horizon is derived. In Section \ref{ss5}, necessary and sufficient conditions for the feedback stabilization of the LFNS in infinite-horizon are derived. Section \ref{ss6} presents its applications in a LF-AUV system. Section \ref{ss7} shows simulation results of the research problem, demonstrating the effectiveness of the optimal decentralized control strategy. Section \ref{ss8} concludes this paper.

{\bf Notation}: $\mathbb{R}^{n}$ denotes the $n$-dimensional Euclidean space; $I_n$ denotes the identity matrix with dimension $n$. $A^\top$ means the transpose of matrix $A$, $A^{-1}$ means the inverse of matrix $A$, and symmetric matrix $A>0$ (or $A\geq0$) means that $A$ is positive definite (or positive semi-definite). $Tr(\cdot)$ represents the trace of a matrix; $\rho(\cdot)$ represents the spectral radius of a matrix. $\mathbb{E}[\cdot]$ represents the mathematical expectation; $\mathbb{E}[\cdot|\sigma(Y)]$ signifies the conditional mathematical expectation with respect to the $\sigma-$ algebra generated by the random vector $Y$. $f(\cdot)$ represents probability density; $f(\cdot|Y)$ signifies conditional probability density with respect to $Y$. Random vector $x\sim \mathcal{N}(\mu, \Sigma)$ means $x$ obeys the normal distribution with mean $\mu$ and covariance $\Sigma$. For the sake of discussion, set $\int_{-\infty}^{+\infty}\cdots\int_{-\infty}^{+\infty}$ as $\int_{\mathbb{R}^{n}}$, and the mathematical expectation of the random vector $X$ with dimension $n$ is $\mathbb{E}[X(t)]=\int_{\mathbb{R}^{n}}Xf(X)dX$. $||X||$ denotes the Euclidean 2-norm of vector $X$. $|y|$ represents the absolute value of the variable $y$.

\section{Problem Formulation and Preliminaries}\label{ss2}

\subsection{Problem Formulation}
In this paper, we shall investigate the following LFNS:
\begin{equation}\label{sta1}
\left\{ \begin{array}{ll}
 x_0(k+1) & =A_{00}x_0(k)+B_{00}u_0(k)+w_0(k),\\
  x_1(k+1) &= A_{11}x_1(k)+B_{11}u_1(k)+w_1(k)\\
&+A_{10}x_0(k)+B_{10}u_0(k),
\end{array} \right.
\end{equation}
in which the integer $k$ is the time instant, $k=0,1,2,\cdots $. $x_0(k)\in\mathbb{R}^{n}$ and $x_1(k)\in\mathbb{R}^{n}$, represent the state information at time $k$ of the leader agent and the follower agent, respectively. $u_0(k)\in\mathbb{R}^{m_1}$ and $u_1(k)\in\mathbb{R}^{m_2}$ are the control inputs of the leader agent and the follower agent, respectively. $A_{00}, A_{10}, A_{11}, B_{00}, B_{10}$ and $B_{11}$ are constant matrices with appropriate dimensions. Both $w_0(k)$ and $w_1(k)$ are the Gaussian white noises satisfying $w_i(k)\sim\mathcal{N}(0,\Sigma_{w_i})$, and the initial states $x_0(0)$ and $x_1(0)$ satisfy $x_i(0)\sim\mathcal{N}(\Bar{x}_i, \Sigma_{x_i})$, $i=0,1$. Without loss of generality, the initial states and the noise vectors at each step $\{x_0(0), x_1(0), w_0(0), \cdots, w_0(k), w_1(0), \cdots, w_1(k)\}$ are all assumed to be mutually independent.

For the LFNS \eqref{sta1}, it can be seen that the state information and the control inputs of the leader agent will affect the follower agent, while the opposite is not true. In other words, for the considered scenario, the information sets accessed by the leader agent and the follower agent can be given as follows:
\begin{equation*}
\left\{ \begin{array}{ll}
\mathcal{S}_0(k)&=\{x_0(0),\cdots,x_0(k), u_0(0),\cdots,u_0(k-1)\},\notag\\
  \mathcal{S}_1(k)&=\{x_0(0),\cdots,x_0(k), u_0(0),\cdots,u_0(k-1),\notag\\
   &~~x_1(0),\cdots,x_1(k), u_1(0),\cdots,u_1(k-1)\}.
\end{array} \right.
\end{equation*}
Let us denote $\mathcal{F}_0(k), \mathcal{F}_1(k)$ as the $\sigma$-algebras generated by $\mathcal{S}_0(k), \mathcal{S}_1(k)$, respectively.

Consequently, in this paper, we assume that $u_0(k)$ and $u_1(k)$ satisfy the following assumption:

\begin{assumption}\label{ass1}
  $u_i(k)$ is $\mathcal{F}_i(k)$-measurable, $i=0,1$ and $\sum_{k=0}^{N}\mathbb{E}[u_i(k)^\top u_i(k)]<\infty$.
\end{assumption}

Corresponding with the LFNS \eqref{sta1}, the following quadratic cost function is introduced:
\begin{align}\label{cost1}
J_N&=\mathbb{E}\bigg[\sum_{k=0}^{N}[X^{\top}(k)QX(k)+ U^{\top}(k)RU(k)]\bigg]\notag\\
&+\mathbb{E}[X(N+1)^{\top}P(N+1)X(N+1)],
\end{align}
where $X(k)=[x_0^{\top}(k)~x_1^{\top}(k)]^\top$, $U(k)=[u_0^{\top}(k)~u_1^{\top}(k)]^\top$, $Q=\left[\hspace{-1mm}
  \begin{array}{cccc}
    Q_{00}&Q_{01}\\
    Q_{10}&Q_{11}
  \end{array}
\hspace{-1mm}\right]$, $R=\left[\hspace{-1mm}
  \begin{array}{cccc}
    R_{00}&R_{01}\\
    R_{10}&R_{11}
  \end{array}
\hspace{-1mm}\right]$, $P(N+1)=\left[\hspace{-1mm}
  \begin{array}{cccc}
    P_{00}(N+1)&P_{01}(N+1)\\
    P_{10}(N+1)&P_{11}(N+1)
  \end{array}
\hspace{-1mm}\right]$.

For the weighting matrices in \eqref{cost1}, we make the following standard assumption, \cite{s2002}.
\begin{assumption}\label{ass2}
  $Q\geq 0, R>0$, $P(N+1)\geq 0$.
\end{assumption}

Now, we will introduce the main problem to be addressed in this paper.

\begin{problem}\label{prob1}
Suppose Assumptions \ref{ass1}-\ref{ass2} hold, find $u_i(k), i=0,1$ to minimize cost function \eqref{cost1}.
\end{problem}

\begin{remark}
The decentralized control problem for the LFNS \eqref{sta1} under asymmetric information (Problem \ref{prob1}) remains unresolved, as prior studies (e.g., \cite{hc1972,sl2010,nkj2018,am2022}) have yet to derive an iterative optimal estimator for the asymmetric information structure, a feasible globally optimal control strategy under leader-follower asymmetry, or necessary and sufficient feedback stabilization conditions for infinite-horizon LFNSs. These gaps stem from the inherent coupling between the leader and the follower stochastic dynamics and the lack of a unified framework to decouple estimation and control under heterogeneous information. Addressing these challenges constitutes the core contribution of this paper.
\end{remark}

\subsection{Solvability of Problem \ref{prob1}}
In this section, we will give the solvability conditions for Problem \ref{prob1}, which serves as preliminaries.

In the first place, to facilitate the discussion, we rewrite \eqref{sta1} as the following compact form:
\begin{align}\label{sta2}
X(k+1) & =AX(k)+BU(k)+W(k),
\end{align}
in which $A=\left[\hspace{-1mm}
  \begin{array}{ccc}
    A_{00}& 0\\
    A_{10}& A_{11}\\
  \end{array}
\hspace{-1mm}\right], B= \left[\hspace{-1mm}
  \begin{array}{ccc}
    B_{00} & 0\\
    B_{10} & B_{11}\\
  \end{array}
\hspace{-1mm}\right], W(k)= \left[\hspace{-1mm}
  \begin{array}{ccc}
    w_0(k)\\
    w_1(k)\\
  \end{array}
\hspace{-1mm}\right] $, and $W(k)\sim\mathcal{N}(0,\Sigma_{W})$.

Next, define the input embedding matrices
\[
D_0=\begin{bmatrix}I_{m_1}\\0\end{bmatrix},\qquad
D_1=\begin{bmatrix}0\\I_{m_2}\end{bmatrix},
\]
so that $u_i(k)=D_i^{\top}U(k), i=0,1$ and $U(k)=D_0u_0(k)+D_1u_1(k)$. In the following, we will present the results for the solvability of Problem \ref{prob1}.

\begin{lemma}\label{lem1}
Under Assumptions \ref{ass1} and \ref{ass2}, Problem \ref{prob1} can be uniquely solved if and only if the stationary condition
\begin{align}\label{sc1}
0&=\mathbb{E}\!\left[D_i^{\top}RU(k)+(BD_i)^{\top}\Theta(k)\mid\mathcal{F}_i(k)\right].
\end{align}
can be uniquely solved. Moreover, the costate $\Theta(k)$ satisfies the following backward difference equation:
\begin{align}\label{costa1}
\Theta(k-1)=\mathbb{E}[A^{\top}\Theta(k)|\mathcal{F}_1(k)]+QX(k),
\end{align}
with the terminal condition
\begin{align}\label{tc}
\Theta(N)=P(N+1)X(N+1).
\end{align}
\end{lemma}

\begin{proof}
 To avoid repetition, we omit the detailed proof, which can be induced from \textit{Theorem 1} of \cite{qxzl2025}.
\end{proof}

\begin{remark}\label{rem1}
It is noted that the LFNS \eqref{sta1} is a forward iterative difference equation, while the costate equation \eqref{costa1} is a backward iterative difference equation. Together with the stationary condition \eqref{sc1}, they form the following FBSDEs:
\begin{equation}\label{FBSDEs}
\left\{ \begin{array}{ll}
 x_0(k+1) & =A_{00}x_0(k)+B_{00}u_0(k)+w_0(k),\\
  x_1(k+1) &= A_{11}x_1(k)+B_{11}u_1(k)+w_1(k)\\
&+A_{10}x_0(k)+B_{10}u_0(k),\\
\Theta(k-1)&=\mathbb{E}[A^{\top}\Theta(k)|\mathcal{F}_1(k)]+QX(k),\\
\Theta(N)&=P(N+1)X(N+1),\\
0&=\mathbb{E}\!\left[D_i^{\top}RU(k)+(BD_i)^{\top}\Theta(k)\mid\mathcal{F}_i(k)\right].
\end{array} \right.
\end{equation}
Therefore, it can be concluded from Lemma \ref{lem1} that solving Problem \ref{prob1} is equivalent to solving the FBSDEs \eqref{FBSDEs}.
\end{remark}

\section{Optimal Decentralized Estimation}\label{ss3}

From Lemma \ref{lem1} and Remark \ref{rem1}, it can be seen that obtaining the optimal control strategy requires calculating a series of conditional mathematical expectations from a probabilistic perspective. This motivates us to introduce relevant results for calculating conditional expectations in this section. In the following, we will give the algorithms of computing the conditional expectation, which is equivalent to solving a type of optimal estimation problems.

For the sake of discussion, we introduce the following symbols ($k=0,\cdots,N, i=0, 1$) :
\begin{equation}\label{est1}
\left\{ \begin{array}{ll}
\hat{x}_i(k|k)&=\mathbb{E}[x_i(k)|\mathcal{F}_0(k)],\\
\tilde{x}_i(k|k)&=x_i(k)-\hat{x}_i(k|k),\\
\hat{x}_i(k+1|k)&=\mathbb{E}[x_i(k+1)|\mathcal{F}_0(k)],\\
\tilde{x}_i(k+1|k)&={x}_i(k+1)-\hat{x}_i(k+1|k),\\
\hat{u}_i(k)&=\mathbb{E}[u_i(k)|\mathcal{F}_0(k)],\\
\tilde{u}_i(k)&=u_i(k)-\hat{u}_i(k).
\end{array} \right.
\end{equation}

Before presenting the main results of this section, we will introduce an important result on the conditional independence in the following lemma.

\begin{lemma}\label{lem2}
Under Assumption \ref{ass1}, $x_0(k)$ and $x_1(k)$ are conditionally independent given $\mathcal{F}_0(k-1)$, $k=1,\cdots, N$. In other words, the following relationship holds:
\begin{align}\label{est2}
\hspace{-2mm}\mathbb{E}[x_0(k)x_1^{\top}(k)|\mathcal{F}_0(k-1)]=\hat{x}_0(k|k-1)\hat{x}_1^{\top}(k|k-1).
\end{align}
\end{lemma}

\begin{proof}
First, we will prove that the states $x_0(k)$ and $x_1(k)$ are conditionally uncorrelated given $\mathcal{F}_0(k-1)$.

Using \eqref{sta1} and \eqref{est1}, for $k=1,\cdots,N,$ we have
\begin{equation}\label{est4}
\left\{ \begin{array}{ll}
\hat{x}_1(k|k-1)&=A_{11}\hat{x}_1(k-1|k-1)+B_{11}\hat{u}_1(k-1)\\
&+A_{10}x_0(k-1)+B_{10}u_0(k-1),\\
\hat{x}_0(k|k-1)&=A_{00}x_0(k-1)+B_{00}u_0(k-1),\\
\tilde{x}_1(k|k-1)&=A_{11}\tilde{x}_1(k-1|k-1)+B_{11}\tilde{u}_1(k-1)\\
&+w_1(k-1),\\
\tilde{x}_0(k|k-1)&=w_0(k-1).
\end{array} \right.
\end{equation}

From \eqref{sta1}, $x_0(k)$, $x_1(k)$, $\{w_0(k)\}_{k=0}^{N}$, $\{w_1(k)\}_{k=0}^{N},$ $k=0,\cdots,N,$ are independent of each other and follow the Gaussian distribution. Accordingly, we have
\begin{align*}
&\mathbb{E}[\tilde{x}_0(k|k-1)\tilde{x}_1^{\top}(k|k-1)|\mathcal{F}_0(k-1)]=0.
\end{align*}

Using \eqref{est1}, there holds:
\begin{align*}
&\mathbb{E}[\tilde{x}_0(k|k-1)\tilde{x}_1^{\top}(k|k-1)|\mathcal{F}_0(k-1)]\notag\\
&=\mathbb{E}\big[[x_0(k)-\hat{x}_0(k|k-1)][x_1(k)\notag\\
&-\hat{x}_1(k|k-1)]^{\top}|\mathcal{F}_0(k-1)\big]\notag\\
&=\mathbb{E}[x_0(k)x_1^{\top}(k)-x_0(k)\hat{x}_1^{\top}(k|k-1)\notag\\
&-\hat{x}_0(k|k-1)x_1^{\top}(k)\notag\\
&+\hat{x}_0(k|k-1)\hat{x}_1^{\top}(k|k-1)|\mathcal{F}_0(k-1)]\notag\\
&=\mathbb{E}[x_0(k)x_1^{\top}(k)|\mathcal{F}_0(k-1)]\notag\\
&-\mathbb{E}[x_0(k)\hat{x}_1^{\top}(k|k-1)|\mathcal{F}_0(k-1)]\notag\\
&-\mathbb{E}[\hat{x}_0(k|k-1)x_ 1^{\top}(k)|\mathcal{F}_0(k-1)]\notag\\
&+\mathbb{E}[\hat{x}_0(k|k-1)\hat{x}_1^{\top}(k|k-1)|\mathcal{F}_0(k-1)].
\end{align*}

From \eqref{est1}, using the properties of conditional mathematical expectation, we have
\begin{align*}
&\mathbb{E}[\tilde{x}_0(k|k-1)\tilde{x}_1^{\top}(k|k-1)|\mathcal{F}_0(k-1)]\notag\\
&=\mathbb{E}[x_0(k)x_1^{\top}(k)|\mathcal{F}_0(k-1)]\notag\\
&-\mathbb{E}[x_0(k)|\mathcal{F}_0(k-1)]\hat{x}_1^{\top}(k|k-1)\notag\\
&-\hat{x}_0(k|k-1)\mathbb{E}[x_1^{\top}(k)|\mathcal{F}_0(k-1)]\notag\\
&+\hat{x}_0(k|k-1)\hat{x}_1^{\top}(k|k-1)\notag\\
&=\mathbb{E}[x_0(k)x_1^{\top}(k)|\mathcal{F}_0(k-1)]-\hat{x}_0(k|k-1)\hat{x}_1^{\top}(k|k-1)\notag\\
&-\hat{x}_0(k|k-1)\hat{x}_1^{\top}(k|k-1)+\hat{x}_0(k|k-1)\hat{x}_1^{\top}(k|k-1)\notag\\
&=\mathbb{E}[x_0(k)x_1^{\top}(k)|\mathcal{F}_0(k-1)]-\hat{x}_0(k|k-1)\hat{x}_1^{\top}(k|k-1).
\end{align*}
Then, \eqref{est2} can be derived. And we can know that the states $x_0(k)$ and $x_1(k)$ are conditionally uncorrelated given $\mathcal{F}_0(k-1)$.

Moreover, we know $x_0(k)$ and $x_1(k)$ follow the
Gaussian distribution. Hence the states $x_0(k)$ and $x_1(k)$ are conditionally independent given $\mathcal{F}_0(k-1)$.
\end{proof}

Consequently, we will introduce a preliminary result, which shows that the conditional expectation is essentially the optimal estimation in the sense of minimizing the mean square error (MSE) covariance.

\begin{lemma} \label{lem3} 
Let $\mathcal{X}\in \mathbb{R}^{n_1}$  and $\mathcal{Y}\in \mathbb{R}^{n_2}$ be random vectors, where $\mathcal{Y}$ is a noisy measurement of $\mathcal{X}$. The optimal estimate of $\mathcal{X}$ given $\mathcal{Y}$, in the sense of minimizing the MSE, is the conditional expectation of $\mathcal{X}$ given $\mathcal{Y}$, denoted by $ \mathbb{E}[\mathcal{X} | \mathcal{Y}]$. That is,
\begin{align*}
\hat{\mathcal{X}} = \mathbb{E}[\mathcal{X} | \mathcal{Y}].
\end{align*}
This optimal estimate minimizes the MSE, which is defined as:
\begin{align*}
\text{MSE}(\hat{\mathcal{X}}) = \mathbb{E} \left[ \|\mathcal{X} - \hat{\mathcal{X}}(\mathcal{Y}) \|^2 \right],
\end{align*}
where $\hat{\mathcal{X}}(\mathcal{Y})$ is any estimate of $\mathcal{X}$ based on the observation $\mathcal{Y} = y$, and the estimate that minimizes the MSE is $\hat{\mathcal{X}}(\mathcal{Y}) = \mathbb{E}[\mathcal{X} | \mathcal{Y}]$.
\end{lemma}

\begin{proof}
A comprehensive proof can refer to \textit{Theorem 3.1, Section 2.3} of \cite{am1979}.
\end{proof}

According to Lemma \ref{lem3}, computing the optimal estimation in the sense of minimizing the MSE is equivalent to calculating the conditional expectation. Next, we will use the conditional independence proven in Lemma \ref{lem2} to solve for the recursive $\hat{x}_i(k|k)$, which is presented in the following theorem.

\begin{theorem}\label{th1}
Under Assumption \ref{ass1}, the recursive optimal estimation (conditional expectation) can be calculated by
\begin{align}\label{est5}
\hat{x}_1(k|k)&=A_{11}\hat{x}_1(k-1|k-1)+B_{11}\hat{u}_1(k-1)\notag\\
&+A_{10}{x}_0(k-1)+B_{10}{u}_0(k-1),
\end{align}
with the initial condition $\hat{x}_1(0|0)=\Bar{x}_1$.
\end{theorem}

\begin{proof}
We will adopt the induction method to derive the main results.

Firstly, noting that the initial states $x_0(0)\sim\mathcal{N}(\bar{x}_0,\Sigma_{x_0})$, $x_1(0)\sim\mathcal{N}(\bar{x}_1,\Sigma_{x_1})$, and $x_0(0)$, $x_1(0)$ are independent of each other, then there follows $\hat{x}_1(0|0)=\bar{x}_1.$

Using the LFNS \eqref{sta1} and \eqref{est1}, we have
\begin{align}\label{est6}
 \hat{x}_1(1|0)&=A_{11}\hat{x}_1(0|0)+B_{11}\hat{u}_1(0)\notag\\
 &+A_{10}x_0(0)+B_{10}u_0(0).
\end{align}

In the following, using \eqref{est1}, from the definition of the conditional expectation, we know $\hat{x}_1(1|1)$ can be calculated as:
\begin{align}\label{est7}
\hat{x}_1(1|1)
&=\mathbb{E}[x_1(1)|\mathcal{F}_0(1)]\notag\\
&=\mathbb{E}[x_1(1)|x_0(1),x_0(0)]\notag\\
&=\int_{\mathbb{R}^n}x_1(1) f \big(x_1(1)|x_0(1),x_0(0)\big)\ dx_1(1),
\end{align}
in which $f \big(x_1(1)|x_0(1),x_0(0)\big)$ is the conditional probability density of $x_1(1)$ given $x_0(1),x_0(0)$.

Moreover, according to Lemma \ref{lem2}, $x_0(1)$ and $x_1(1)$ are conditionally independent given $\mathcal{F}_0(0)$, then we have
\begin{align}\label{est8}
f\big(x_1(1)|x_0(1), x_0(0)\big)=&\frac{f\big(x_1(1), x_0(1), x_0(0)\big)}{f\big(x_0(1), x_0(0)\big)}\notag\\
=&\frac{f \big(x_1(1), x_0(1)|x_0(0)\big)}{f \big(x_0(1)|x_0(0)\big)}\notag\\
=&\frac{f\big(x_1(1)|x_0(0)\big)f\big(x_0(1)|x_0(0)\big)}{f \big(x_0(1)| x_0(0)\big)}\notag\\
=&f \big(x_1(1)|x_0(0)\big).
\end{align}

Thus, by using \eqref{est8}, it can be derived from \eqref{est7} that
\begin{align}\label{est9}
\hat{x}_1(1|1)&=\int_{\mathbb{R}^n} x_1(1) f \big(x_1(1)|x_0(0)\big)\ dx_1(1)\notag\\
&=\mathbb{E}[x_1(1)|\mathcal{F}_0(0)]\notag\\
&=\hat{x}_1(1|0).
\end{align}

Therefore, combining \eqref{est6} and \eqref{est9},  \eqref{est5} has been verified for $k=1$.

In order to use the induction method, we assume that \eqref{est5} holds for $\hat{x}_1(k|k), k=1,\cdots,l$, next we will prove that $\hat{x}_1(l+1|l+1)$ also satisfies \eqref{est5}.

In fact, from the system equation \eqref{sta1}, we have
\begin{align}\label{est10}
\hat{x}_1(l+1|l)&=A_{11}\hat{x}_1(l|l)+B_{11}\hat{u}_1(l)\notag\\
&+A_{10}x_0(l)+B_{10}u_0(l).
\end{align}

In the following, using \eqref{est1}, by following the lines of \eqref{est7}, $\hat{x}_1(l+1|l+1)$ is given by
\begin{align}\label{est11}
\hat{x}_1(l+1|l+1)
&=\mathbb{E}[x_1(l+1)|\mathcal{F}_0(l+1)]\notag\\
&=\mathbb{E}[x_1(l+1)|x_0(l+1),x_0(l)]\notag\\
&=\int_{\mathbb{R}^n} x_1(l+1) f \big(x_1(l+1)|x_0(l)\big)\ dx_1(l+1)\notag\\
&=\hat{x}_1(l+1|l).
\end{align}

Hence, combining \eqref{est10} and \eqref{est11}, we can conclude that \eqref{est5} can be proved for $k=l+1$, which completes the proof.
\end{proof}

\begin{remark}\label{rem3}
 In fact, according to Theorem \ref{th1}, using \eqref{sta1}, \eqref{sta2} and \eqref{est1}, for $k \geq 1,$ \eqref{est5} can be rewritten as the following compact form:
 \begin{equation}\label{est12}
\left\{ \begin{array}{ll}
\hat{X}(k|k)&=A\hat{X}(k-1|k-1)+B\hat{U}(k-1)\\
&+\mathbf{I}_0W(k-1),\\
\tilde{X}(k|k)&=A\tilde{X}(k-1|k-1)+B\tilde{U}(k-1)\\
&+\mathbf{I}_1W(k-1),
\end{array} \right.
\end{equation}
where $\hat{X}(k|k)= \left[\hspace{-1mm}
  \begin{array}{ccc}
    {x}_0(k)\\
     \hat{x}_1(k|k)\\
  \end{array}
\hspace{-1mm}\right]$ with the initial condition $\hat{X}(0|0)= \left[\hspace{-1mm}
  \begin{array}{ccc}
    {x}_0(0)\\
     \Bar{x}_1\\
  \end{array}
\hspace{-1mm}\right]$, $\tilde{X}(k|k)=\left[\hspace{-1mm}
  \begin{array}{ccc}
    0\\
    \tilde{x}_1(k|k)\\
  \end{array}
\hspace{-1mm}\right]$, $\hat{U}(k)= \left[\hspace{-1mm}
  \begin{array}{ccc}
    u_0(k)\\
     \hat{u}_1(k)\\
  \end{array}
\hspace{-1mm}\right]$, $\tilde{U}(k)=U(k)-\hat{U}(k)$, $k=0,\cdots, N,$ $\mathbf{I}_0= \left[\hspace{-1mm}
  \begin{array}{ccc}
    I_n&0\\
    0&0\\
  \end{array}
\hspace{-1mm}\right]$, $\mathbf{I}_1= \left[\hspace{-1mm}
  \begin{array}{ccc}
    0&0\\
    0&I_n\\
  \end{array}
\hspace{-1mm}\right]$.
\end{remark}

\begin{remark}\label{rem4}
Compared with the estimators proposed in previous works \cite{sl2010,nkj2018}, which cannot be computed recursively, an iterative form of the optimal estimation for Problem \ref{prob1} is firstly derived in Theorem \ref{th1}, which is crucial for solving the FBSDEs \eqref{FBSDEs}. 
\end{remark}

\section{Finite-Horizon Optimal Decentralized Control}\label{ss4}

In this section, the optimal decentralized control strategy in finite-horizon will be derived via decoupling the FBSDEs \eqref{FBSDEs} by the use of orthogonal decomposition method.

\begin{theorem}\label{th2}
Under Assumptions \ref{ass1} and \ref{ass2}, the optimal control strategy $\big(u_0(k),u_1(k)\big),k=0,\cdots,N$, of minimizing cost function \eqref{cost1} can be given by
\begin{align}\label{contr1}
u_0(k)&=-K_{00}(k)x_0(k)-K_{01}(k)\hat x_1(k|k),
\end{align}
\begin{align}\label{contr2}
u_1(k)&=-K_{10}(k)x_0(k)-K_{11}(k)\hat x_1(k|k)-G(k)\tilde x_1(k|k),
\end{align}
in which $K_{ij}(k), i,j=0,1$ is the block matrix of $K(k)$ with $K(k)=\left[\hspace{-1mm}
  \begin{array}{cccc}
    K_{00}(k)&K_{01}(k)\\
    K_{10}(k)&K_{11}(k)
  \end{array}
\hspace{-1mm}\right]$, satisfying
\begin{equation}\label{gain1}
\left\{ \begin{array}{ll}
K(k)=&\Lambda^{-1}(k)L(k),\\
L(k)=&B^{\top}P(k+1)A,\\
\Lambda(k)=&R+B^{\top}P(k+1)B,\\
P(k)=&Q+A^{\top}P(k+1)A\\
&-L^{\top}(k)\Lambda^{-1}(k)L(k),~~ P(N+1).
\end{array} \right.
\end{equation}
The gain matrix $G(k)$ and private Riccati matrix $S(k)$ satisfy
\begin{equation}\label{finite:privateRiccati}
\left\{
\begin{aligned}
G(k)&=\Phi^{-1}(k)\Xi(k),\\
\Xi(k)&=B_{11}^{\top}S(k+1)A_{11},\\
\Phi(k)&=R_{11}+B_{11}^{\top}S(k+1)B_{11},\\
S(k)&=Q_{11}+A_{11}^{\top}S(k+1)A_{11}\\
&\quad-\Xi^{\top}(k)\Phi^{-1}(k)\Xi(k),~~S(N+1)=P_{11}(N+1).
\end{aligned}
\right.
\end{equation}

Furthermore, the relationship between the state $X(k)$ and the costate $\Theta(k)$ is given by
\begin{equation}\label{costa2}
\mathbb E[\Theta(k-1)\mid\mathcal F_0(k)]
=P(k)\hat X(k|k),
\end{equation}
\begin{equation}\label{costa3}
\begin{aligned}
\Theta_1(k-1)=P_{10}(k)x_0(k)+P_{11}(k)\hat x_1(k|k)+S(k)\tilde x_1(k|k),
\end{aligned}
\end{equation}
where $\Theta(k-1)=[\Theta_0(k-1)^{\top}, \Theta_1(k-1)^{\top}]^{\top}$.

Meanwhile, the associated optimal cost function $J^{*}_N$ is calculated as:
\begin{equation}\label{cost2}
\begin{aligned}
J^{*}_N&=\mathbb E\!\left[
\hat X^\top(0|0)P(0)\hat X(0|0)\right]+\operatorname{Tr}\!\left(S(0)\Sigma_{x_1}\right)\\
&+\sum_{k=0}^{N}\operatorname{Tr}\!\big(
P_{00}(k+1)\Sigma_{w_0}+S(k+1)\Sigma_{w_1}\big).
\end{aligned}
\end{equation}
\end{theorem}

\begin{proof}
According to Assumption \ref{ass1}, by the tower property of conditional expectation \cite{d2019}, in order to ensure the adaptability of $u_0(k)$ and $u_1(k)$, we rewrite \eqref{sc1} as:
\begin{equation}\label{sc2}
0=R\hat U(k)+B^\top\mathbb E[\Theta(k)\mid\mathcal F_0(k)],
\end{equation}
\begin{equation}\label{sc3}
\begin{aligned}
0=R_{11}\tilde u_1(k)+B_{11}^\top\big[\mathbb E[\Theta_1(k)\mid\mathcal F_1(k)]-\mathbb E[\Theta_1(k)\mid\mathcal F_0(k)]\big].
\end{aligned}
\end{equation}
These conditions are equivalent to \eqref{sc1}, since $BD_1=[0,B_{11}^\top]^\top$ and $\tilde U(k)=D_1\tilde u_1(k)$.

Using backward induction method, firstly, with $k=N$, using \eqref{sta2} and \eqref{tc}, it can be implied from \eqref{sc2} that
\begin{align*}
0&=R\hat U(N)+B^\top P(N+1)A\hat X(N|N)\\
&+B^\top P(N+1)B\hat U(N)\\
&=\Lambda(N)\hat U(N)+L(N)\hat X(N|N),
\end{align*}
where $\Lambda(N)$ and $L(N)$ are given by \eqref{gain1} at $k=N$. 

Noting that, from Assumption \ref{ass2}, it can be concluded that $\Lambda(N)$ is positive definite, then we have
\begin{equation}\label{contr3}
\begin{aligned}
\hat U(N)&=-\Lambda^{-1}(N)L(N)\hat X(N|N)\\
&=-K(N)\hat X(N|N).
\end{aligned}
\end{equation}

Using \eqref{sta2} and \eqref{tc}, there holds:
\[
\begin{aligned}
&\mathbb E[\Theta_1(N)\mid\mathcal F_1(N)]
-\mathbb E[\Theta_1(N)\mid\mathcal F_0(N)]\\
&=P_{11}(N+1)\big(A_{11}\tilde x_1(N|N)
+B_{11}\tilde u_1(N)\big).
\end{aligned}
\]
Since $S(N+1)=P_{11}(N+1)$, using \eqref{sc3}, we have
\[
\begin{aligned}
0=\Phi(N)\tilde u_1(N)+\Xi(N)\tilde x_1(N|N),
\end{aligned}
\]
where $\Phi(N)$ and $\Xi(N)$ are defined in \eqref{finite:privateRiccati}. 

Since $S(N+1)\geq0$ and $R_{11}>0$, $\Phi(N)$ is positive definite, and hence
\begin{equation}\label{contr4}
\begin{aligned}
\tilde u_1(N)&=-\Phi^{-1}(N)\Xi(N)\tilde x_1(N|N)\\
&=-G(N)\tilde x_1(N|N),
\end{aligned}
\end{equation}
where $G(N)$ is shown in \eqref{finite:privateRiccati}.

Together with $\tilde U(N)=D_1\tilde u_1(N)$, \eqref{contr3}--\eqref{contr4} give \eqref{contr1}--\eqref{contr2} at $k=N$.

In the following, we will calculate $\Theta(N-1)$. Actually, using \eqref{sta2}, the terminal condition \eqref{tc}, \eqref{contr3} and \eqref{contr4}, it follows from \eqref{costa1} with $k=N$ that
\begin{align*}
\Theta(N-1)
&=QX(N)+A^\top P(N+1)[AX(N)+BU(N)]\\
&=P(N)\hat X(N|N)+Q\tilde X(N|N)\\
&+A^\top P(N+1)[A\tilde X(N|N)
+B\tilde U(N)].
\end{align*}
Taking the conditional expectation of both sides of the above equation given $\mathcal F_0(N)$, yields
\[
\mathbb E[\Theta(N-1)\mid\mathcal F_0(N)]
=P(N)\hat X(N|N).
\]

Noting that $\tilde u_0(N)=0$, there holds
\[
A\tilde X(N|N)+B\tilde U(N)
=
\begin{bmatrix}
0\\
A_{11}\tilde x_1(N|N)+B_{11}\tilde u_1(N)
\end{bmatrix}.
\]
Since $\Theta(k-1)=[\Theta_0(k-1)^{\top}, \Theta_1(k-1)^{\top}]^{\top}$, we have
\[
\begin{aligned}
&\Theta_1(N-1)\\
&=P_{10}(N)x_0(N)+P_{11}(N)\hat x_1(N|N)+Q_{11}\tilde x_1(N|N)\\
&+A_{11}^\top P_{11}(N+1)[A_{11}\tilde x_1(N|N)+B_{11}\tilde u_1(N)].
\end{aligned}
\]
Using $P_{11}(N+1)=S(N+1)$ and \eqref{contr4}, it can be obtained that
\[
\begin{aligned}
&\Theta_1(N-1)\\
&=P_{10}(N)x_0(N)+P_{11}(N)\hat x_1(N|N)\\
&+\big[Q_{11}+A_{11}^\top S(N+1)A_{11}-\Xi^\top(N)\Phi^{-1}(N)\Xi(N)\big]\tilde x_1(N|N)\\
&=P_{10}(N)x_0(N)+P_{11}(N)\hat x_1(N|N)+S(N)\tilde x_1(N|N).
\end{aligned}
\]
Thus \eqref{costa2}--\eqref{costa3} hold at $k=N$. 

Consequently, to use backward induction, we assume that for arbitrary $l$, and for $k=l+1,\cdots,N$, there holds:
\begin{itemize}
  \item [1)] The optimal control pair $\big(u_0(k), u_1(k)\big)$ is given by \eqref{contr1}-\eqref{contr2};
  \item [2)] The costate $\Theta(k-1)$ satisfies \eqref{costa2} and \eqref{costa3}.
\end{itemize}

Then, taking the conditional expectation of \eqref{costa2} at time $l+1$ given $\mathcal F_0(l)$, and using \eqref{est12}, yields
\begin{equation}\label{finite:adjointMeans}
\begin{aligned}
\mathbb E[\Theta(l)\mid\mathcal F_0(l)]=P(l+1)[A\hat X(l|l)+B\hat U(l)].
\end{aligned}
\end{equation}
Taking the conditional expectation of \eqref{costa3} at time $l+1$ given $\mathcal F_0(l)$ and $\mathcal F_1(l)$, and using \eqref{est4}, we have
\begin{equation}\label{finite:privateAdjointMeans}
\begin{aligned}
&\mathbb E[\Theta_1(l)\mid\mathcal F_1(l)]
-\mathbb E[\Theta_1(l)\mid\mathcal F_0(l)]\\
&=S(l+1)[A_{11}\tilde x_1(l|l)+B_{11}\tilde u_1(l)].
\end{aligned}
\end{equation}
Substituting \eqref{finite:adjointMeans} and \eqref{finite:privateAdjointMeans} into \eqref{sc2}--\eqref{sc3} gives
\[
\begin{aligned}
0&=\Lambda(l)\hat U(l)+L(l)\hat X(l|l),\\
0&=\Phi(l)\tilde u_1(l)+\Xi(l)\tilde x_1(l|l).
\end{aligned}
\]

According to \textit{Lemma 6} of \cite{qxzl2025}, it can be deduced from Assumption \ref{ass2} that $P(l+1)\geq 0, S(l+1)\geq 0 $, hence $\Lambda(l)>0, \Phi(l)>0$ can be derived. In this case, the optimal control strategies can be calculated as follows,
\begin{align}
\hat U(l)&=-K(l)\hat X(l|l),\label{contr5}\\
\tilde u_1(l)&=-G(l)\tilde x_1(l|l).\label{contr6}
\end{align}

Next, taking the common conditional expectation of \eqref{costa1} and substituting \eqref{finite:adjointMeans} and \eqref{contr5} gives
\[
\begin{aligned}
\mathbb E[\Theta(l-1)\mid\mathcal F_0(l)]&=\big[Q+A^\top P(l+1)(A-BK(l))\big]\hat X(l|l)\\
&=P(l)\hat X(l|l).
\end{aligned}
\]
Subtracting the follower block of this common relation from the follower block of \eqref{costa1}, and using \eqref{finite:privateAdjointMeans} and \eqref{contr6}, yields
\[
\begin{aligned}
&\Theta_1(l-1)-\mathbb E[\Theta_1(l-1)\mid\mathcal F_0(l)]\\
&=\big[Q_{11}+A_{11}^\top S(l+1)A_{11}-\Xi^\top(l)\Phi^{-1}(l)\Xi(l)\big]\tilde x_1(l|l)\\
&=S(l)\tilde x_1(l|l).
\end{aligned}
\]
Then, there holds:
\[
\begin{aligned}
\Theta_1(l-1)=P_{10}(l)x_0(l)+P_{11}(l)\hat x_1(l|l)+S(l)\tilde x_1(l|l).
\end{aligned}
\]

To end the proof, we will calculate the optimal cost function $J_N^*$ in \eqref{cost2}. In fact, by using \eqref{sta2}, \eqref{sc1}, \eqref{costa1}, \eqref{costa2} and \eqref{costa3}, it yields that
\[
\begin{aligned}
&\mathbb E\big[X^\top(k)\Theta(k-1)-X^\top(k+1)\Theta(k)\big]\\
&=\mathbb E\big[X^\top(k)QX(k)+U^\top(k)RU(k)-W^\top(k)\Theta(k)\big].
\end{aligned}
\]
Summing from $k=0$ to $k=N$ and using \eqref{tc} yields
\begin{equation}\label{sum2}
J_N^*=\mathbb E[X^\top(0)\Theta(-1)]
+\sum_{k=0}^N\mathbb E[W^\top(k)\Theta(k)].
\end{equation}
Using \eqref{costa2} at time $k+1$ for the leader noise and \eqref{costa3} for the follower noise, there holds
\[
\begin{aligned}
\mathbb E[w_0^\top(k)\Theta_0(k)]
&=\operatorname{Tr}\!\big(P_{00}(k+1)\Sigma_{w_0}\big),\\
\mathbb E[w_1^\top(k)\Theta_1(k)]
&=\operatorname{Tr}\!\big(S(k+1)\Sigma_{w_1}\big).
\end{aligned}
\]
Similarly, the initial costate relations and conditional centering give
\[
\begin{aligned}
&\mathbb E[X^\top(0)\Theta(-1)]\\
&=\mathbb E[\hat X^\top(0|0)\Theta(-1)]+\mathbb E[\tilde x_1^\top(0|0)\Theta_1(-1)]\\
&=\mathbb E\!\left[
\mathbb E[\hat X^\top(0|0)\Theta(-1)\mid\mathcal F_0(0)]
\right]+\mathbb E[\tilde x_1^\top(0|0)S(0)\tilde x_1(0|0)]\\
&+\mathbb E\!\left[
\mathbb E[\tilde x_1(0|0)\mid\mathcal F_0(0)]^\top
\big(P_{10}(0)x_0(0)+P_{11}(0)\hat x_1(0|0)\big)
\right]\\
&=\mathbb E\!\left[
\hat X^\top(0|0)
\mathbb E[\Theta(-1)\mid\mathcal F_0(0)]
\right]+\mathbb E[\tilde x_1^\top(0|0)S(0)\tilde x_1(0|0)]\\
&=\mathbb E\big[\hat X^\top(0|0)P(0)\hat X(0|0)\big]
+\operatorname{Tr}\!\big(S(0)\Sigma_{x_1}\big).
\end{aligned}
\]
Substituting these identities into \eqref{sum2} proves \eqref{cost2} and completes the proof.
\end{proof}

\begin{remark}
It can be observed from Theorem \ref{th2} and its proof that the key to decoupling and solving the FBSDEs \eqref{FBSDEs} lies in utilizing the orthogonality between $\hat{u}_1(k)$ and $\tilde{u}_1(k)$, which we refer to as the ``orthogonal decomposition method". Corresponding to this orthogonal decomposition approach, the form of the controller in the optimal control strategies \eqref{contr1} and \eqref{contr2} are also divided into two parts, with $\hat{x}_1(k|k)$, $\tilde{x}_1(k|k)$ being orthogonal. Furthermore, compared to previous results, this paper does not give a predefined form for the optimal control strategy. Instead, the optimal control strategy is directly derived by solving the FBSDEs \eqref{FBSDEs}.
\end{remark}

\begin{remark}
The common and private gains $K(k)$ and $G(k)$ are determined by different Riccati recursions, and generally $G(k)\neq K_{11}(k)$. The original full costate is retained in \eqref{costa1}--\eqref{tc}; identities \eqref{costa2}--\eqref{costa3} describe its common projection and follower private projection.
\end{remark}

\section{Infinite-Horizon Optimal Decentralized Feedback Stabilization Control}\label{ss5}

\subsection{Problem Formulation}
In this section, we will solve the optimal decentralized feedback stabilization control problem of infinite-horizon. For this purpose, we introduce the following infinite-horizon cost function:
\begin{align}\label{cost3}
\mathcal{J}=\mathbb{E}\bigg[\sum_{k=0}^{\infty}\gamma^k[X^{\top}(k)QX(k)+ U^{\top}(k)RU(k)]\bigg],
\end{align}
where $0<\gamma<1$ is a given discount factor.

Firstly, we introduce the following definitions.

\begin{definition}\label{def1}
The LFNS \eqref{sta1} with $(u_0(k),u_1(k))=(0,0)$ is called mean square stable (MSS) if there exists constant $c_i>0$ independent of the initial states $x_0(0),x_1(0)$ such that $\lim_{k\to\infty}\mathbb{E}[x_i(k)]=0, i=0,1 $ and $\lim_{k\to\infty}\mathbb{E}[x_i^{\top}(k)x_i(k)]=c_i$.
\end{definition}

\begin{definition}\label{def2}
Under Assumption \ref{ass1}, the LFNS \eqref{sta1} is called feedback stabilizable if there exists $\mathcal{F}_0(k)-$ adapted $u_0(k)$ and $\mathcal{F}_1(k)-$ adapted $u_1(k)$ such that the closed-loop LFNS \eqref{sta1} is MSS.
\end{definition}

Without loss of generality, the following two assumptions are made, \cite{iyb2006}.
\begin{assumption}\label{ass3}
$R>0$ and $Q = C^{\top}C\geq0$ for some matrix $C$.
\end{assumption}

\begin{assumption}\label{ass4}
$(A, Q^{1/2})$ is observable.
\end{assumption}

Next, the decentralized feedback stabilization control problem to be addressed in this section is stated below:

\begin{problem}\label{prob2}
Suppose Assumptions \ref{ass1}, \ref{ass3}-\ref{ass4} hold, find the necessary and sufficient conditions to make the LFNS \eqref{sta1} feedback stabilizable. Meanwhile, minimize \eqref{cost3} subject to \eqref{sta2}.
\end{problem}

\subsection{Main Results}

For the sake of discussion, we will first introduce the cost function in finite-horizon as follows:
 \begin{align}\label{cost4}
\mathcal{J}_N=\mathbb{E}\bigg[\sum_{k=0}^{N}\gamma^k[X^{\top}(k)QX(k)+U^{\top}(k)RU(k)]\bigg],
\end{align}
in which the coefficients $\gamma, Q, R$ are the same as those in \eqref{cost3}.

By using the results of Theorem \ref{th2}, we shall give the following lemma without proof.

\begin{lemma}\label{lem4}
Under Assumptions \ref{ass1} and \ref{ass3}, the unique optimal control pair $\big(u_0(k),u_1(k)\big)$ minimizing \eqref{cost4}, for $k=0,\ldots,N$, is given by
\begin{equation}\label{contr7}
u_0(k)=-H_{N00}(k)x_0(k)-H_{N01}(k)\hat x_1(k|k),
\end{equation}
\begin{equation}\label{contr8}
\begin{aligned}
u_1(k)&=-H_{N10}(k)x_0(k)-H_{N11}(k)\hat x_1(k|k)-\Omega_N(k)\tilde x_1(k|k),
\end{aligned}
\end{equation}
in which the gain matrices $H_{Nij}(k), i,j=0,1$ is the block matrix of
$H_N(k)=\begin{bmatrix}
H_{N00}(k)&H_{N01}(k)\\
H_{N10}(k)&H_{N11}(k)
\end{bmatrix}$, satisfying
\begin{equation}\label{gain2}
\left\{\begin{aligned}
H_N(k)&=\gamma\Psi_N^{-1}(k)L_N(k),\\
L_N(k)&=B^\top\mathcal P_N(k+1)A,\\
\Psi_N(k)&=R+\gamma B^\top\mathcal P_N(k+1)B,
\end{aligned}\right.
\end{equation}
and $\mathcal{P}_N(k)$ satisfies the following common Riccati equation:
\begin{equation}\label{ricc1}
\left\{\begin{aligned}
\mathcal P_N(k)
&=Q+\gamma A^\top\mathcal P_N(k+1)A\\
&-\gamma^2L_N^\top(k)\Psi_N^{-1}(k)L_N(k),~\mathcal P_N(N+1)=0.
\end{aligned}\right.
\end{equation}
The private gain $\Omega_N(k)$ satisfies
\begin{equation}\label{discount:privateRiccati}
\left\{\begin{aligned}
\Omega_N(k)&=\gamma\Pi_N^{-1}(k)\Delta_N(k),\\
\Delta_N(k)&=B_{11}^\top\mathsf{S}_N(k+1)A_{11},\\
\Pi_N(k)&=R_{11}+\gamma B_{11}^\top\mathsf{S}_N(k+1)B_{11},
\end{aligned}\right.
\end{equation}
and $\mathsf{S}_N(k)$ satisfies the following private Riccati equation:
\begin{equation}\label{ricc22}
\left\{\begin{aligned}
\mathsf{S}_N(k)&=Q_{11}+\gamma A_{11}^\top\mathsf{S}_N(k+1)A_{11}\\
&-\gamma^2\Delta_N^\top(k)\Pi_N^{-1}(k)\Delta_N(k),~\mathsf{S}_N(N+1)=0.
\end{aligned}\right.
\end{equation}
The estimate $\hat x_1(k|k)$ is calculated from \eqref{est5}, and $\tilde x_1(k|k)$ is defined in \eqref{est1}.

Moreover, the minimum cost is
\begin{equation}\label{cost5}
\begin{aligned}
\mathcal J_N^*
&=\mathbb E\!\left[\hat X^\top(0|0)\mathcal P_N(0)\hat X(0|0)\right]+\operatorname{Tr}\!\big(\mathsf{S}_N(0)\Sigma_{x_1}\big)\\
&+\sum_{k=0}^{N}\gamma^{k+1}\operatorname{Tr}\!\big(
\mathcal P_{N00}(k+1)\Sigma_{w_0}+\mathsf{S}_N(k+1)\Sigma_{w_1}\big).
\end{aligned}
\end{equation}
\end{lemma}

Based on the above conclusions, we will present the following main result.
\begin{theorem}\label{th3}
Under Assumptions \ref{ass1}, \ref{ass3} and \ref{ass4}, the LFNS \eqref{sta1} can be feedback stabilizable by the following control strategy:
\begin{equation}\label{contr9}
\left\{ \begin{array}{ll}
u_0(k)=&-H_{00}x_0(k)-H_{01}\hat{x}_1(k|k),\\
u_1(k)=&-H_{10}x_0(k)-H_{11}\hat{x}_1(k)-\Omega\tilde x_1(k|k),
\end{array} \right.
\end{equation}
where $\hat{x}_1(k|k)$ can be calculated from \eqref{est5}, if and only if the following two conditions hold:
\begin{itemize}
  \item [1)]  There exists the unique solutions $\mathcal{P}>0$ and $\mathsf{S}>0$ satisfying the following  Riccati equations:
\begin{align}\label{ricc2}
 \mathcal{P}=Q+\gamma A^{\top}\mathcal{P}A-\gamma^2L^{\top}\Psi^{-1}L,
\end{align}
\begin{align}\label{ricc33}
 \mathsf{S}=Q_{11}+\gamma A_{11}^\top\mathsf{S}A_{11}-\gamma^2\Delta^\top\Pi^{-1}\Delta,
\end{align}
  \item [2)] \begin{align}\label{ricc3}
(1-\gamma)\mathcal{P}< Q+H^{\top}RH,
\end{align}
\begin{align}\label{ricc44}
(1-\gamma)\mathsf{S}< Q_{11}+\Omega^{\top}R_{11}\Omega,
\end{align}
where 
\begin{equation}\label{gain3}
\left\{ \begin{array}{ll}
H&=\left[\hspace{-1mm}
  \begin{array}{cccc}
    H_{00}&H_{01}\\
    H_{10}&H_{11}
  \end{array}
\hspace{-1mm}\right],\\
H&=\gamma \Psi^{-1}L,\\
L&=B^{\top}\mathcal{P}A,\\
\Psi&=R+\gamma B^{\top}\mathcal{P}B,\\
\mathcal{P}&=Q+\gamma A^{\top}\mathcal{P}A-\gamma^2L^{\top}\Psi^{-1}L,
\end{array} \right.
\end{equation}
and
\begin{equation}
\left\{
\begin{aligned}
\Omega&=\gamma\Pi^{-1}\Delta,\\
\Delta&=B_{11}^\top\mathsf SA_{11},\\
\Pi&=R_{11}+\gamma B_{11}^\top\mathsf SB_{11},\\
\mathsf S
&=Q_{11}+\gamma A_{11}^\top\mathsf SA_{11}
-\gamma^2\Delta^\top\Pi^{-1}\Delta.
\end{aligned}
\right.
\end{equation}
\end{itemize}
Meanwhile, the cost function in infinite-horizon \eqref{cost3} is minimized by \eqref{contr9} as follows:
\begin{equation}
\begin{aligned}\label{cost6}
\mathcal J^*
&=\mathbb E[\hat X^\top(0|0)\mathcal P\hat X(0|0)]
+\operatorname{Tr}(\mathsf S\Sigma_{x_1})\\
&\quad+\frac{\gamma}{1-\gamma}
\operatorname{Tr}\!\big(
\mathcal P_{00}\Sigma_{w_0}+\mathsf S\Sigma_{w_1}\big).
\end{aligned}
\end{equation}
\end{theorem}

\begin{proof}
`Necessity': In the first place, we will prove that $\mathcal{P}_N(k)$ in \eqref{ricc1} and $\mathsf{S}_N(k)$ in \eqref{ricc22} are convergent with $k\rightarrow+\infty$. From Lemma \ref{lem4}, if we set $\Sigma_{w_0}=\Sigma_{w_1}=0$, since $\hat{X}(0|0)$ and $\Sigma_{x_1}$ are arbitrary, we have $\mathcal{P}_N(0)\leq \mathcal{P}_{N+1}(0)$ and $\mathsf{S}_N(0)\leq \mathsf{S}_{N+1}(0)$. Furthermore, noting that the coefficient matrices of \eqref{sta1} and \eqref{cost4} are time-invariant, then we have
\begin{align}\label{ati}
  \mathcal{P}_N(k)=\mathcal{P}_{N-k}(0), \mathsf{S}_N(k)=\mathsf{S}_{N-k}(0)
\end{align}
Combining \eqref{ati}, we can conclude that $\mathcal{P}_k(0)$ $\mathsf{S}_k(0)$ are monotonically non-decreasing sequence with respect to $k$.

Next, we will show $\mathcal{P}_k(0)$ and $\mathsf{S}_k(0)$ are bounded. From Definition \ref{def2}, there exist constant matrices $\hat{K}, \tilde{K}$ such that the LFNS \eqref{sta1} is feedback stabilizable with control strategies $\hat{U}(k)=\hat{K}\hat{X}(k|k)$ and $\tilde{u}_1(k)=\tilde{K}\tilde{x}_1(k|k)$. Under Assumptions \ref{ass1} and \ref{ass3}, using \eqref{cost3}, select a constant $c_2(c_2>0)$, such that $Q\leq c_2I$, $Q_{11}\leq c_2I$, ${\hat{K}}^{\top}R\hat{K}\leq c_2I$ and ${\tilde{K}}^{\top}R_{11}\tilde{K}\leq c_2I$. Then, we have
{\small\begin{align*}
\mathcal{J}
&=\mathbb{E}\bigg[\sum_{k=0}^{\infty}\gamma^k[\hat{X}^{\top}(k|k)(Q+{\hat{K}}^{\top}R\hat{K})\hat{X}(k|k)\notag\\
&+\tilde{x}_1^{\top}(k|k)(Q_{11}+{\tilde{K}}^{\top}R_{11}\tilde{K})\tilde{x}_1(k|k)]\bigg]\notag\\
&\leq2c_2\mathbb{E}\bigg[\sum_{k=0}^{\infty}\gamma^k[\hat{X}^{\top}(k|k)\hat{X}(k|k)+\tilde{x}_1^{\top}(k|k)\tilde{x}_1(k|k)]\bigg].
\end{align*}}

According to \eqref{est12}, it can be obtained that
\begin{equation}\label{sta3}
\hspace{-4mm}\left\{ \begin{array}{ll}
&\mathbb{E}[\hat{X}^{\top}(k|k)\hat{X}(k|k)]\\
&=\mathbb{E}\big[\hat{X}^{\top}(0|0)[(A+B\hat{K})^k]^{\top}(A+B\hat{K})^k\hat{X}(0|0)\big]\\
&+\sum_{i=0}^{k-1}Tr\big(\mathbf{I}_0\Sigma_W[(A+B\hat{K})^i]^{\top}(A+B\hat{K})^i\big),\\
&\mathbb{E}[\tilde{x}_1^{\top}(k|k)\tilde{x}_1(k|k)]\\
&=\mathbb{E}\big[\tilde{x}_1^{\top}(0|0)[(A_{11}+B_{11}\tilde{K})^k]^{\top}(A_{11}+B_{11}\tilde{K})^k\tilde{x}_1(0|0)\big]\\
&+\sum_{i=0}^{k-1}Tr\big(\Sigma_{w_1}[(A_{11}+B_{11}\tilde{K})^i]^{\top}(A_{11}+B_{11}\tilde{K})^i\big).
\end{array} \right.
\end{equation}

Combining Definition \ref{def1} and \ref{def2}, using \eqref{est12}, there exist constants $c_4>0$ and $c_5>0$ following: $\lim_{k\to\infty}\mathbb{E}[\hat{X}^{\top}(k|k)\hat{X}(k|k)]=c_4, \lim_{k\to\infty}\mathbb{E}[\tilde{x}_1^{\top}(k|k)\tilde{x}_1(k|k)]=c_5$. From \eqref{sta3}, we can conclude that $\mathbb{E}\big[\hat{X}^{\top}(0|0)[(A+B\hat{K})^k]^{\top}(A+B\hat{K})^k\hat{X}(0|0)\big]$ and $\mathbb{E}\big[\sum_{i=0}^{k-1}Tr\big(\mathbf{I}_0\Sigma_W[(A+B\hat{K})^i]^{\top}(A+B\hat{K})^i\big)\big]$ are convergent with $k\rightarrow+\infty$. Hence, there holds $\rho(A+B\hat{K})<1$. Similarly, $\rho(A_{11}+B_{11}\tilde{K})<1$ can be derived.

Using  \eqref{sta3}, it can be obtained that:
\begin{equation*}
\left\{ \begin{array}{ll}
\mathbb{E}[\hat{X}^{\top}(k|k)\hat{X}(k|k)]\leq\mathbb{E}[\hat{X}^{\top}(0|0)\hat{X}(0|0)]+c_4,\\
\mathbb{E}[\tilde{x}_1^{\top}(k|k)\tilde{x}_1(k|k)]\leq\mathbb{E}[\tilde{x}_1^{\top}(0|0)\tilde{x}_1(0|0)]+c_5.
\end{array} \right.
\end{equation*}
Then, the following relationship holds:
\begin{align*}
\mathcal{J}_N^{*}\leq \mathcal{J}_N\leq \mathcal{J}\leq \frac{2c_2}{1-\gamma}\mathbb{E}[X^{\top}(0)X(0)+c_4+c_5].
\end{align*}

It can be deduced from \eqref{cost5} that $\mathcal{P}_k(0)$ and $\mathsf{S}_k(0)$ are bounded. 

In the following, we will show $\mathcal{P}>0$ and $\mathsf{S}>0$ by contradiction. Otherwise, we assume that there exists nonzero $\hat{X}(0|0)=[x_0^{\top}(0)~\hat{x}_1^{\top}(0|0)]^\top\neq 0$ and $\tilde{x}_1(0|0)\neq 0$ such that $\hat{X}^\top(0|0)\mathcal{P}\hat{X}(0|0)=0$ and $\tilde{x}_1^\top(0|0)\mathsf{S}\tilde{x}_1(0|0)=0$. Setting $\Sigma_W=0$, using \eqref{cost4} and \eqref{cost5}, we can conclude that 
\begin{align*}
\mathcal{J}_N^*=\hat{X}^\top(0|0)\mathcal{P}\hat{X}(0|0)+\tilde{x}_1^\top(0|0)\mathsf{S}\tilde{x}_1(0|0)=0,
\end{align*}
which indicates
\begin{align}\label{fff}
CX(0)=0, U(k)=0, k\geq0.
\end{align}
in which Assumption \ref{ass3} has been inserted.

Noting Assumption \ref{ass4}, $(A,C)$ is observable. It can be obtained from \eqref{fff} that $\hat{X}(0|0)=0, \tilde{x}_1(0|0)=0$, which contradicts with $\hat{X}(0|0)\neq0$ and $\tilde{x}_1(0|0)\neq 0$. Therefore, $\lim_{k\to\infty} \mathcal{P}_k(0)=\mathcal{P}>0$ and $\lim_{k\to\infty} \mathsf{S}_k(0)=\mathsf{S}>0$ are shown.

In the following, we will show the uniqueness of the solution $\mathcal{P}$ and $\mathsf{S}$. Considering $\Sigma_W=0$, let $\tilde{\mathcal{P}}$ and $\tilde{\mathsf{S}}$ are another solutions to \eqref{ricc2} and \eqref{ricc33} satisfying $\tilde{\mathcal{P}}>0$, $\tilde{\mathsf{S}}>0$ and $\tilde{\mathcal{P}}\neq\mathcal{P}$, $\tilde{\mathsf{S}}\neq\mathsf{S}$. From \eqref{cost5}, the optimal value of the cost function in finite-horizon is as:
\begin{align*}
\mathcal{J}_N^*&=\hat{X}^\top(0|0)\mathcal{P}\hat{X}(0|0)+\tilde{x}_1^\top(0|0)\mathsf{S}\tilde{x}_1(0|0)\\
&=\hat{X}^\top(0|0)\tilde{\mathcal{P}}\hat{X}(0|0)+\tilde{x}_1^\top(0|0)\tilde{\mathsf{S}}\tilde{x}_1(0|0).
\end{align*}

Since $\hat{X}(0|0)$ and $\tilde{x}_1(0|0)$ are arbitrary, considering $\hat{X}(0|0)\neq0$ and $\tilde{x}_1(0|0)\neq0$, we will obtain that $\mathcal{P}=\tilde{\mathcal{P}}$ and $\mathsf{S}=\tilde{\mathsf{S}}$, which contradicts with $\tilde{\mathcal{P}}\neq\mathcal{P}$ and $\mathsf{S}\neq\tilde{\mathsf{S}}$. Now, the uniqueness of $\mathcal{P}$ and $\mathsf{S}$ can be shown.

Finally, \eqref{ricc3} and \eqref{ricc44} will be shown. Using \eqref{ricc2} and \eqref{ricc33}, we can obtain that $\mathcal{P}$ and $\mathsf{S}$ satisfy
\begin{align}\label{lyp1}
 \mathcal{P}&=\gamma(A-BH)^{\top}\mathcal{P}(A-BH)+Q+H^{\top}RH,\notag\\
 \mathsf{S}&=\gamma(A_{11}-B_{11}\Omega)^{\top}\mathsf{S}(A_{11}-B_{11}\Omega)+Q_{11}+\Omega^{\top}R_{11}\Omega.
\end{align}
 Combining  \eqref{contr7} and \eqref{contr8}, \eqref{contr9} can be derived with $k\rightarrow+\infty$. In this case, using \eqref{contr9}, it can be obtain $\hat{K}=-H$ and $\tilde{K}=-\Omega$. Hence, there holds $\rho(A-BH)<1$ and $\rho(A_{11}-B_{11}\Omega)<1$. Then, \eqref{ricc3} and \eqref{ricc44} can be derived from \eqref{lyp1}. The proof of the necessity is complete.

`Sufficiency': using \eqref{lyp1}, there holds:
\begin{align}\label{lyp2}
 \mathcal{P}&=(A-BH)^{\top}\mathcal{P}(A-BH)+F_0,\notag\\
 \mathsf{S}&=(A_{11}-B_{11}\Omega)^{\top}\mathsf{S}(A_{11}-B_{11}\Omega)+F_1,
\end{align}
where $F_0=\frac{Q+H^{\top}RH-(1-\gamma)\mathcal{P}}{\gamma}$ and $F_1=\frac{Q_{11}+\Omega^{\top}R_{11}\Omega-(1-\gamma)\mathsf{S}}{\gamma}$. Using \eqref{ricc3}, we can conclude $F_0>0$ and $F_1>0$. It can be derived from \eqref{lyp2} that
\begin{align}\label{lyp3}
\mathcal{P}&>(A-BH)^{\top}\mathcal{P}(A-BH),\notag\\
\mathsf{S}&>(A_{11}-B_{11}\Omega)^{\top}\mathsf{S}(A_{11}-B_{11}\Omega).
\end{align}

For the sake of discussion, we denote
\begin{align}\label{sta5}
 \xi_0(k+1)&=(A-BH)\xi_0(k),\notag\\
 \xi_1(k+1)&=(A_{11}-B_{11}\Omega)\xi_1(k).
\end{align}
Then we have
\begin{align}\label{sta6}
\xi_0^{\top}(k)\xi_0(k)&=\xi_0^{\top}(0)[(A-BH)^k]^{\top}(A-BH)^k\xi_0(0),\notag\\
\xi_1^{\top}(k)\xi_1(k)&=\xi_1^{\top}(0)[(A_{11}-B_{11}\Omega)^k]^{\top}(A_{11}-B_{11}\Omega)^k\xi_1(0).
\end{align}

In the following, we will prove that $\xi_0^{\top}(k)\xi_0(k)$ and $\xi_1^{\top}(k)\xi_1(k)$ are convergent. From \eqref{lyp3}, we have that
\begin{align*}
 \xi_0^{\top}(k)\mathcal{P}\xi_0(k)&>\xi_0^{\top}(k)(A-BH)^{\top}\mathcal{P}(A-BH)\xi_0(k),\notag\\
\xi_1^{\top}(k)\mathsf{S}\xi_1(k)&>\xi_1^{\top}(k)(A_{11}-B_{11}\Omega)^{\top}\mathsf{S}(A_{11}-B_{11}\Omega)\xi_1(k).
\end{align*}
Then, using \eqref{sta5} there holds:
\begin{align*}
\xi_0^{\top}(k)\xi_0(k)&>\xi_0^{\top}(k+1)\xi_0(k+1),\notag\\
\xi_1^{\top}(k)\xi_1(k)&>\xi_1^{\top}(k+1)\xi_1(k+1).
\end{align*}
Obviously, $\xi_i^{\top}(k)\xi_i(k), i=0,1$ monotonically decreases with $k$. Noting $\xi_i^{\top}(k)\xi_i(k)\geq0$. we can conclude that $\xi_i^{\top}(k)\xi_i(k)$ is convergent with $k\rightarrow+\infty$.

Consequently, we shall prove that the LFNS \eqref{sta1} can be feedback stabilizable by \eqref{contr9}. 

Noting $\xi_i^{\top}(k)\xi_i(k)$ is convergent with $k\rightarrow+\infty$, from \eqref{sta6}, we have $\lim_{k\to\infty}(A-BH)^k=0$ and $\lim_{k\to\infty}(A_{11}-B_{11}\Omega)^k=0$, i.e., $\rho(A-BH)<1$ and $\rho(A_{11}-B_{11}\Omega)<1$. Using \eqref{sta2}, it can be obtained that
{\small\begin{equation}\label{sta4}
\left\{ \begin{array}{ll}
&\mathbb{E}[X(k)]=(A-BH)^k\mathbb{E}[\hat{X}(0|0)]+[0,I_n]^{\top}(A_{11}-B_{11}\Omega)^k\mathbb{E}[\tilde{x}_1(0|0)],\\
&\mathbb{E}[X^{\top}(k)X(k)]=\mathbb{E}\big[\hat{X}^{\top}(0|0)[(A-BH)^k]^{\top}(A-BH)^k\hat{X}(0|0)\big]\\
&+\sum_{i=0}^{k-1}Tr\big(\mathbf{I}_0\Sigma_W[(A-BH)^i]^{\top}(A-BH)^i\big)\\
&+\mathbb{E}\big[\tilde{x}_1^{\top}(0|0)[(A_{11}-B_{11}\Omega)^k]^{\top}(A_{11}-B_{11}\Omega)^k\tilde{x}_1(0|0)\big]\\
&+\sum_{i=0}^{k-1}Tr\big(\Sigma_{w_1}[(A_{11}-B_{11}\Omega)^i]^{\top}(A_{11}-B_{11}\Omega)^i\big).
\end{array} \right.
\end{equation}}
From \eqref{sta4}, we have
\begin{align}\label{eee}
&\lim_{k\to\infty}\mathbb{E}[X(k)]=0.
\end{align}

Now, we will show $\lim_{k\to\infty}\mathbb{E}[X^{\top}(k)X(k)]<\infty$. For the sake of discussion, we define the Lyapunov function candidate $V(k)$ as:
\begin{align}\label{lyp4}
V(k)&=\gamma^k\mathbb{E}[\hat{X}^{\top}(k|k)\mathcal{P}\hat{X}(k|k)]+\gamma^k\mathbb{E}[\tilde{x}_1^{\top}(k|k)\mathsf{S}\tilde{x}_1(k|k)]\notag\\
&+\sum_{i=k}^{\infty}\gamma^{i+1}\operatorname{Tr}\!\big(
\mathcal P_{00}\Sigma_{w_0}+\mathsf{S}\Sigma_{w_1}\big).
\end{align}
Then, using \eqref{sta2}, the control strategy \eqref{contr9}, Riccati equations \eqref{ricc2}, \eqref{ricc33} and the Lyapunov function \eqref{lyp4}, we have
\begin{align}\label{lyap2}
V(k)-V(k+1)=\gamma^k\mathbb{E}[X^{\top}(k)QX(k)+U^{\top}(k)RU(k)]\geq0.
\end{align}
Obviously, $V(k)$ is monotonically non-increasing with respect to $k$, $V(k)\geq0$. Hence, $V(k)$ is convergent with $k$.

From \eqref{lyp4}, we know $\gamma^k\mathbb{E}[\hat{X}^{\top}(k|k)\mathcal{P}\hat{X}(k|k)]$ and $\gamma^k\mathbb{E}[\tilde{x}_1^{\top}(k|k)\mathsf{S}\tilde{x}_1(k|k)]$ are convergent with $k$. Substituting \eqref{contr9} into \eqref{sta2}, we have
\begin{align}\label{ppp}
&\gamma^k\mathbb{E}[\hat{X}^{\top}(k|k)\mathcal{P}\hat{X}(k|k)]\notag\\
&=\gamma^k\mathbb{E}\big[\hat{X}^{\top}(0|0)[(A-BH)^k]^{\top}\mathcal{P}(A-BH)^k\hat{X}(0|0)\big]\notag\\
&+\gamma^k\sum_{i=0}^{k-1}Tr\big(\mathbf{I}_0\Sigma_W[(A-BH)^i]^{\top}\mathcal{P}(A-BH)^i\big),\notag\\
&\gamma^k\mathbb{E}[\tilde{x}_1^{\top}(k|k)\mathsf{S}\tilde{x}_1(k|k)]\notag\\
&=\gamma^k\mathbb{E}\big[\tilde{x}_1^{\top}(0|0)[(A_{11}-B_{11}\Omega)^k]^{\top}\mathsf{S}(A_{11}-B_{11}\Omega)^k\tilde{x}_1(0|0)\big]\notag\\
&+\gamma^k\sum_{i=0}^{k-1}Tr\big(\Sigma_{w_1}[(A_{11}-B_{11}\Omega)^i]^{\top}\mathsf{S}(A_{11}-B_{11}\Omega)^i\big).
\end{align}
Similarly, from \eqref{ppp}, we know $\gamma^k\sum_{i=0}^{k-1}Tr\big(\mathbf{I}_0\Sigma_W[(A-BH)^i]^{\top}\mathcal{P}(A-BH)^i\big)$ and $\gamma^k\sum_{i=0}^{k-1}Tr\big(\Sigma_{w_1}[(A_{11}-B_{11}\Omega)^i]^{\top}\mathsf{S}(A_{11}-B_{11}\Omega)^i\big)$ are convergent with $k$. Due to $\rho(A-BH)<1$, $\rho(A_{11}-B_{11}\Omega)<1$ and $0<\gamma<1$, it can be obtained that $\sum_{i=0}^{k-1}Tr\big(\mathbf{I}_0\Sigma_W[(A-BH)^i]^{\top}(A-BH)^i\big)$ and $\sum_{i=0}^{k-1}Tr\big(\Sigma_{w_1}[(A_{11}-B_{11}\Omega)^i]^{\top}(A_{11}-B_{11}\Omega)^i\big)$ are convergent with $k$. Then, using \eqref{sta4}, we have
\begin{align}\label{ggg}
 \lim_{k\to\infty}\mathbb{E}[X^{\top}(k)X(k)]=c_6,
\end{align}
where $c_6$ is a constant.

According to \eqref{eee} and \eqref{ggg}, the control strategy in Theorem \ref{th3} make the LFNS \eqref{sta1} feedback stabilizable.  
 
Finally, we will calculate $\mathcal{J}^*$. By taking summation from $k = 0$ to $k = N$ on both sides of \eqref{lyap2}, there holds:
\begin{equation}
\begin{aligned}\label{sum3}
&V(0)-V(N+1)\\
&=\sum_{k=0}^{N}\gamma^k
\mathbb E\!\left[X^\top(k)QX(k)+U^\top(k)RU(k)\right]\\
&=\mathbb E\!\left[
\hat X^\top(0|0)\mathcal P\hat X(0|0)\right]
+\operatorname{Tr}\!\left(\mathsf S\Sigma_{x_1}\right)\\
&+\sum_{k=0}^{N}\gamma^{k+1}
\operatorname{Tr}\!\left(
\mathcal P_{00}\Sigma_{w_0}+\mathsf S\Sigma_{w_1}\right)\\
&-\gamma^{N+1}\mathbb E\!\left[
\hat X^\top(N+1|N+1)\mathcal P\hat X(N+1|N+1)\right.\\
&\left.
+\tilde x_1^\top(N+1|N+1)\mathsf S
\tilde x_1(N+1|N+1)
\right].
\end{aligned}
\end{equation}
From \eqref{ggg}, there holds $\lim_{k\to\infty}\gamma^k\mathbb{E}[\hat{X}^{\top}(k|k)\mathcal{P}\hat{X}(k|k)]=0$ and $\lim_{k\to\infty}\gamma^k\mathbb{E}[\tilde{x}_1^{\top}(k|k)\mathsf{S}\tilde{x}_1(k|k)]=0$. Using \eqref{cost3} and \eqref{sum3}, the optimal cost function \eqref{cost6} has been verified, which ends the sufficiency proof.
\end{proof}

\begin{remark}\label{rem6}
Noting that relationship $\mathbb{E}[X^{\top}(k)X(k)]=\mathbb{E}[\hat{X}^{\top}(k|k)\hat{X}(k|k)]
+\mathbb{E}[\tilde{x}_1^{\top}(k|k)\tilde{x}_1(k|k)]$ holds, then it can be observed from Theorem \ref{th3} and its proof that, under the assumptions of Theorem \ref{th3}, the stability of the optimal estimator (Theorem \ref{th1} and Remark \ref{rem3}) is equivalent to the feedback stabilizability of the LFNS \eqref{sta1}.
\end{remark}

\begin{remark}\label{rem7}
Compared to previous works \cite{hc1972,sl2010, nkj2018, am2022}, feedback stabilization problem for the LFNS \eqref{sta1} is firstly solved in Theorem \ref{th3}, and the necessary and sufficient conditions for its feedback stabilizability are proposed, serving as one of the main contributions and innovations of this paper. 
\end{remark}

\section{Applications in LF-AUV System}\label{ss6}

In this section, we proceed to apply the proposed main theoretical results of Section \ref{ss3}-Section \ref{ss5} to the trajectory tracking control problem of LF-AUV system.

\subsection{Introduction to the AUV Model}
Firstly, for the sake of discussion, we shall  consider a three-DOF AUV model. Assume that the AUV moves in the horizontal plane, and the roll, pitch, and heave are close to zero. The motion of the AUV is described by the surge, sway, and yaw dynamics, represented as, \cite{wsyzj2018,qxfy2018}:
\begin{align}\label{AUV1}
\Dot{\eta}=R(\eta)\nu, ~M\Dot{\nu}+C(\nu)\nu+D(\nu)\nu=\tau_s,
\end{align}
where $\eta=[x,y,\psi]^{\top}\in\mathbb{R}^{3}$ denotes the positions and the yaw angle of the AUV, and $\nu=[u,v,r]^{\top}\in\mathbb{R}^{3}$ denotes the translation velocities and the angular velocity of the AUV. $\tau_s=[\tau_u, \tau_v, \tau_r]^{\top}\in\mathbb{R}^{3}$; $\tau_u$ and $\tau_v$ represent longitudinal (surge) and lateral (sway) thrusts of the AUV, respectively; $\tau_r$ represents the yaw moment of the AUV. $R(\eta)\in\mathbb{R}^{3\times3}$ is the rotation matrix, $M\in\mathbb{R}^{3\times3}$ is the positive-definite inertia matrix, $C(\nu)\in\mathbb{R}^{3\times3}$ is  the Coriolis matrix, and $D(\nu)\in\mathbb{R}^{3\times3}$ is the damping matrix.

Subsequently, by linearizing and discretizing \eqref{AUV1}, the discrete-time three-DOF AUV model can be presented as follows, \cite{w2020}:
\begin{align}\label{AUV4}
z(k+1)=A_zz(k)+B_zu_z(k)+w_z(k),
\end{align}
in which $z(k)=[z^{1{\top}}(k),z^{2{\top}}(k)]^{\top}\in\mathbb{R}^{6}$ represents the tracking error between the AUV and its desired trajectory. The given reference trajectory $\eta_d(k)=[x_d(k),y_d(k),\psi_d(k)]^{\top}\in\mathbb{R}^{3}$ is known to the controller and is twice continuously differentiable. $w_z(k)$ is assumed to be Gaussian white noise, representing the random interference induced by modeling inaccuracies, linearization errors, and environmental disturbances. By setting $T$ as the sampling period of the system \eqref{AUV4}, the coefficient matrices in \eqref{AUV4} are given as follows:
{\small\begin{equation}\label{mmm}
\left\{ \begin{array}{ll}
A_z&=T\left[\hspace{-1mm}
  \begin{array}{ccc}
    0 & I_3\\
    0 & 0
  \end{array}
\hspace{-1mm}\right]+I_6, B_z=T[0,I_3]^{\top},\\
\eta(k)&=[x(k),y(k),\psi(k)]^{\top},\\
z^{1}(k)&=\eta(k)-\eta_d(k),\\
z^{2}(k)&=R(k)\nu(k)-\dot{\eta}_d(k),\\
u_z(k)&=R(k)M^{-1}\tau_s(k)-\ddot{\eta}_d(k)+R(k)M^{-1}h(k),\\
h(k)&=MR^{-1}(k)\dot{R}(k)\nu(k)-C(k)\nu(k)-D(k)\nu(k).
\end{array} \right.
\end{equation}}

\subsection{Decentralized Control of LF-AUV System}

In this section, we will investigate the optimal decentralized feedback stabilization control problem for the following LF-AUV system with asymmetric information structure:
\begin{equation}\label{sta7}
\hspace{-1mm}\left\{ \begin{array}{ll}
z_0(k+1)&=A_{z00}z_0(k)+B_{z00}u_{z0}(k)+w_{z0}(k),\\
z_1(k+1) &= A_{z11}z_1(k)+B_{z11}u_{z1}(k)+w_{z1}(k)\\
 &+A_{z10}z_0(k)+B_{z10}u_{z0}(k),
\end{array} \right.
\end{equation}
where $k$ is the time instant, the states $z_i(k)\in\mathbb{R}^{6}$, $i=0,1$, represent the tracking errors at time $k$ of the leader AUV and the follower AUV, respectively, $u_{zi}(k)\in\mathbb{R}^{3}$, $i=0,1$ are the control inputs of the leader AUV and the follower AUV, and $w_{zi}(k)$, $i=0,1$ are the Gaussian white noises satisfying $w_{zi}(k)\sim\mathcal{N}(0,\Sigma_{w_{zi}})$. Moreover, the physical meanings of the coefficient matrices $A_{z00}, A_{z1i}, B_{z00}, B_{z1i}, i=0,1$ in \eqref{sta7} are the same with those in \eqref{mmm}. The initial states $z_i(0)$ satisfying $z_i(0)\sim\mathcal{N}(\Bar{z}_i, \Sigma_{z_i})$, $i=0,1$. Without loss of generality, the initial states and the noise vectors at each step $\{z_0(0),z_1(0), w_{z0}(0), \cdots, w_{z0}(k), w_{z1}(0), \cdots ,w_{z1}(k)\}$ are all assumed to be mutually independent.

To facilitate the discussion, we rewrite \eqref{sta7} as the following compact form:
\begin{align}\label{sta8}
  Z(k+1) & =A_ZZ(k)+B_ZU_Z(k)+W_Z(k),
\end{align}
where $Z(k)=[z_0^{\top}(k)~z_1^{\top}(k)]^\top$, $U_Z(k)=[u_{z0}^{\top}(k)~u_{z1}^{\top}(k)]^\top$, $A_Z=\left[\hspace{-1mm}
  \begin{array}{ccc}
    A_{z00} & 0\\
    A_{z10} & A_{z11}
  \end{array}
\hspace{-1mm}\right], B_Z= \left[\hspace{-1mm}
  \begin{array}{ccc}
    B_{z00} & 0\\
    B_{z10} & B_{z11}
  \end{array}
\hspace{-1mm}\right]$ and $W_Z(k)= \left[\hspace{-1mm}
  \begin{array}{ccc}
    w_{z0}(k)\\
    w_{z1}(k)
  \end{array}
\hspace{-1mm}\right] $.

As shown in Fig. \ref{0},
\begin{figure}
\centerline{\includegraphics[width=3.0in]{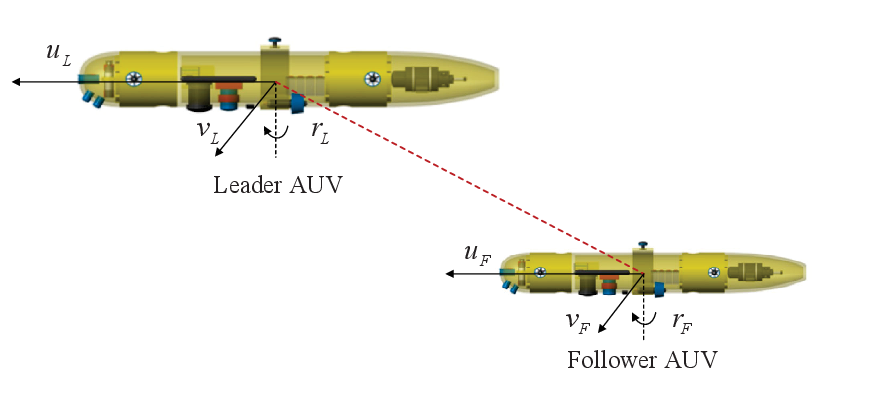}}
\caption{Schematic diagram of leader-follower AUV (LF-AUV) system.}
\label{0}
\end{figure}
for the considered LF-AUV system \eqref{sta7}, it is assumed that the state information and the control inputs of the leader AUV will affect the follower AUV, while the opposite is not true. Hence, for the considered scenario, the information sets accessed by the leader AUV and the follower AUV can be given as follows:
\begin{equation*}
\left\{ \begin{array}{ll}
\mathcal{S}_{z0}(k)&=\{z_0(0),\cdots,z_0(k),u_{z0}(0),\cdots,u_{z0}(k-1)\},\notag\\
  \mathcal{S}_{z1}(k)&=\{z_0(0),\cdots,z_0(k),u_{z0}(0),\cdots,u_{z0}(k-1),\notag\\
   &~~z_1(0),\cdots,z_1(k),u_{z1}(0),\cdots,u_{z1}(k-1)\}.
\end{array} \right.
\end{equation*}
Let us denote $\mathcal{F}_{z0}(k), \mathcal{F}_{z1}(k)$ as the $\sigma$-algebras generated by $\mathcal{S}_{z0}(k), \mathcal{S}_{z1}(k)$, respectively.

We aim to minimize both the tracking error of the LF-AUV system \eqref{sta7} following desired trajectories and the energy consumption cost during the tracking process. Therefore, we introduce the following quadratic cost function:
{\small\begin{align}\label{cost9}
\mathcal{J}_{Z}=\mathbb{E}\bigg[\sum_{k=0}^{\infty}\gamma^k[Z^{\top}(k)Q_ZZ(k)+ U_Z^{\top}(k)R_ZU_Z(k)]\bigg],
\end{align}}
where $Q_Z=\left[\hspace{-1mm}
  \begin{array}{cccc}
    Q_{z00}&Q_{z01}\\
    Q_{z10}&Q_{z11}
  \end{array}
\hspace{-1mm}\right]$, $R_Z=\left[\hspace{-1mm}
  \begin{array}{cccc}
    R_{z00}&R_{z01}\\
    R_{z10}&R_{z11}
  \end{array}
\hspace{-1mm}\right]$. $0<\gamma<1$ is a given discount factor.

Similar to Assumptions \ref{ass1}-\ref{ass4}, we make the following assumptions:
\begin{assumption}\label{ass5}
$u_{zi}(k)$ is $\mathcal{F}_{zi}(k)$-measurable, $i=0,1$ and $\sum_{k=0}^{N}\mathbb{E}[u_{zi}^\top u_{zi}]<\infty$.
\end{assumption}

\begin{assumption}\label{ass6}
$R_Z>0$ and $Q_Z = C_Z^{\top}C_Z\geq0$ for some matrix $C_Z$.
\end{assumption}

\begin{assumption}\label{ass7}
$(A_Z, Q_Z^{1/2})$ is observable.
\end{assumption}

In this section, the following problem will be solved.
\begin{problem}\label{prob4}
Under Assumptions \ref{ass5}-\ref{ass7}, find $u_{zi}(k), i=0,1, k=0,\cdots,N$ to minimize cost function \eqref{cost9}, which makes the LF-AUV system \eqref{sta7} feedback stabilizable.
\end{problem}

At the end of this section, based on the conclusions of Theorems \ref{th1}-\ref{th3}, we provide the solution to Problem \ref{prob4} without proof.
\begin{theorem}\label{th4}
Under Assumptions \ref{ass5}-\ref{ass7}, the LF-AUV \eqref{sta7} can be
feedback stabilizable by the following control strategy:
\begin{equation}\label{contr10}
\left\{ \begin{array}{ll}
u_{z0}(k)&=-H_{z00}z_0(k)-H_{z01}\hat{z}_1(k|k),\\
u_{z1}(k)&=-H_{z10}z_0(k)-H_{z11}\hat{z}_1(k|k)-\Omega_{z}\tilde{z}_1(k|k)
\end{array} \right.
\end{equation}
where the recursive optimal
estimation $\hat{z}_1(k|k)$ as follows:
\begin{align*}
\hat{z}_1(k|k)&=A_{z11}\hat{z}_1(k-1|k-1)+B_{z11}\hat{u}_{z1}(k-1)\notag\\
&+A_{z10}{z}_0(k-1)+B_{z10}{u}_{z0}(k-1),
\end{align*}
if and only if the following two conditions hold:
\begin{itemize}
  \item [1)] There exist the unique
solutions $\mathcal{P}_Z>0$ and $\mathsf{S}_z>0$ satisfying the following  Riccati equation:
\begin{align}\label{ricd}
 \mathcal{P}_Z=Q_Z+\gamma A_Z^{\top}\mathcal{P}_ZA_Z-\gamma^2L_Z^{\top}\Psi_Z^{-1}L_Z,
\end{align}
\begin{align}\label{ricd2}
 \mathsf{S}_z=Q_{z11}+\gamma A_{z11}^\top\mathsf{S}_zA_{z11}-\gamma^2\Delta_z^\top\Pi_z^{-1}\Delta_z,
\end{align}
\item [2)] \begin{align}\label{ricdi}
(1-\gamma)\mathcal{P}_Z< Q_Z+H_Z^{\top}R_ZH_Z,
\end{align}
\begin{align}\label{ricdi2}
(1-\gamma)\mathsf{S}_z< Q_{z11}+\Omega_z^{\top}R_{z11}\Omega_z,
\end{align}
in which the gain matrices $H_{Zij}, i,j=0,1$ is the block matrix of $H_Z=\left[\hspace{-1mm}
  \begin{array}{cccc}
    H_{z00}&H_{z01}\\
    H_{z10}&H_{z11}
  \end{array}\hspace{-1mm}\right]$, satisfying \begin{equation*}
\left\{ \begin{array}{ll}
H_Z&=\gamma \Psi_Z^{-1}L_Z,\\
 L_Z&=B_Z^{\top}\mathcal{P}_ZA_Z,\\
 \Psi_Z&=R_Z+\gamma B_Z^{\top}\mathcal{P}_ZB_Z,
\end{array} \right.
\end{equation*}
and the gain matrix $\Omega_z$, private Riccati matrix $\mathsf S_z$ satisfy
\[
\left\{
\begin{aligned}
\Omega_z&=\gamma\Pi_z^{-1}\Delta_z,\\
\Delta_z&=B_{z11}^\top\mathsf S_zA_{z11},\\
\Pi_z&=R_{z11}+\gamma B_{z11}^\top\mathsf S_zB_{z11}.
\end{aligned}
\right.
\]
\end{itemize}
Noting that $R_0(k), R_1(k)$ are nonsingular, \cite{w2020}, the force vectors $\tau_{si}(k)\in\mathbb{R}^{3}, i=0,1$ of the leader AUV and the follower AUV can be uniquely presented as:
{\small\begin{equation}\label{act}
\left\{ \begin{array}{ll}
\tau_{s0}(k)&=M_0R_0^{-1}(k)[u_{z0}(k)+\ddot{\eta}_{d0}(k)]-h_0(k),\\
\tau_{s1}(k)&=M_1R_1^{-1}(k)[u_{z1}(k)+\ddot{\eta}_{d1}(k)]-h_1(k).
\end{array} \right.
\end{equation}}                                   
In this case, the optimal cost function is given by
\[
\begin{aligned}\label{cost}
\mathcal J_Z^*
&=\mathbb E[\hat Z^\top(0|0)\mathcal P_Z\hat Z(0|0)]
+\operatorname{Tr}(\mathsf S_z\Sigma_{z_1})\\
&\quad+\frac{\gamma}{1-\gamma}
\operatorname{Tr}\!\big(
\mathcal P_{Z00}\Sigma_{w_{z0}}+\mathsf S_z\Sigma_{w_{z1}}\big).
\end{aligned}
\]
\end{theorem}

\begin{proof}
Due to space limitations, the detailed proof is omitted here.
\end{proof}

\section{Numerical Example}\label{ss7}
In this section, we shall employ numerical example to validate the effectiveness of the theoretical results proposed in this paper. Without loss of generality, for the LF-AUV system \eqref{sta7} and cost function \eqref{cost9} in Section~\ref{ss6}, we set $T=1s$, $\gamma=0.92$. The initial errors are Gaussian, and the disturbances satisfy $w_{zi}(k)\sim\mathcal N(0,\Sigma_{w_{zi}})$, $i=0,1$, with the independence properties specified in Section~\ref{ss6}. The complete model parameters, initial distributions, reference trajectories, and noise covariances are provided in the accompanying Python code.

According to the results of Theorem \ref{th4}, the gain matrix of the optimal control strategy $u_{z0}(k), u_{z1}(k)$ in \eqref{contr10}, $\mathcal{P}_Z$ in Riccati equation \eqref{ricd} and $\mathsf S_z$ in Riccati equation \eqref{ricd2} can be calculated. It can be verified that the inequality \eqref{ricdi} and the inequality \eqref{ricdi2} hold. Therefore, according to Theorem \ref{th4}, LF-AUV system \eqref{sta7} can be feedback stabilizable.

Fig. ~\ref{fig:planar} presents the planar trajectories for the same realization as in Fig.~\ref{fig:components}. The leader AUV and follower AUV track their respective sinusoidal reference trajectories while retaining the prescribed relative reference offset. Blue and yellow solid lines indicate the actual trajectories of the leader AUV and follower AUV, respectively; red and black dashed lines indicate their desired trajectories. The visible deviations reflect the persistent random disturbances.
\begin{figure}[!htbp]
\centering
\includegraphics[width=\columnwidth]{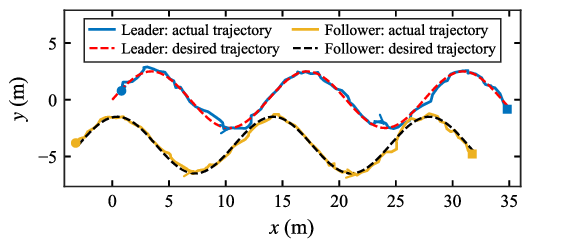}
\caption{Simulation curves for the actual and the desired planar trajectories of leader AUV and follower AUV.}
\label{fig:planar}
\end{figure}

Tracking accuracy is further quantified by the root mean square error (RMSE). The position RMSE uses the planar squared distance $(x_i-x_{di})^2+(y_i-y_{di})^2$, and the heading RMSE uses $(\psi_i-\psi_{di})^2$. Table~\ref{tab:trackingRMSE} reports the results for the full $0$--$100\,\mathrm{s}$ interval and the $20$--$100\,\mathrm{s}$ post-transient interval. The nonzero RMSEs reflect fluctuations caused by persistent process noise.

\begin{table}[!htbp]
\caption{Position and heading tracking RMSEs of leader AUV and follower AUV.}
\label{tab:trackingRMSE}
\centering
\footnotesize
\setlength{\tabcolsep}{4pt}
\renewcommand{\arraystretch}{1.15}
\begin{tabular}{@{}llrr@{}}
\toprule
AUV & \shortstack{Interval (s)} & \shortstack{Position RMSE (m)} & \shortstack{Heading RMSE (rad)}\\
\midrule
Leader & 0--100 & 0.3269 & 0.02461 \\
Follower & 0--100 & 0.3146 & 0.02408 \\
\midrule
Leader & 20--100 & 0.3064 & 0.02229 \\
Follower & 20--100 & 0.3039 & 0.02227 \\
\bottomrule
\end{tabular}
\end{table}

In fact, substituting the optimal control strategy $u_{zi}(k), i=0, 1$ into \eqref{sta7}, the tracking errors in $x, y, \psi$ of the leader AUV and the follower AUV can be calculated and the results are shown in Fig. 4 (a) and Fig. 4 (b). Additionally, based on the tracking errors calculated above and the desired trajectories, the actual $x, y, \psi$ of the leader AUV and the follower AUV can be calculated and the results are shown in Fig. 5 (a) and Fig. 5 (b). As shown in Fig.~\ref{fig:tracking}, the ensemble-mean tracking errors of both AUVs decay toward zero under the prescribed Gaussian disturbances. In addition, according to \eqref{act}, the force vectors $\tau_{si}(k)=[\tau_{ui}(k), \tau_{vi}(k), \tau_{ri}(k)]^\top, i=0,1$ of the leader AUV and follower AUV can be calculated, as shown in Fig. 6 (a) and Fig. 6 (b).

Finally, a finite time horizon $N$ is selected. Using the coefficients given above, the optimal cost function can be calculated from Theorem \ref{th2}. The relationship between this optimal cost function and the time horizon $N$ is illustrated in Fig. 7 (a). As shown in Fig. 7 (a), the optimal cost function converges as $N$ increases, and the converged value coincides exactly with the optimal value of the LF-AUV system's cost function, i.e., $\mathcal{J}_Z^{*}=10980.47$, as calculated from Theorem \ref{th4}.

Moreover, the additional performance index of system's cost with force and moment penalties is
\begin{equation}\label{sim:physicalCost}
\begin{aligned}
\mathcal J_{\tau,N}
&=\mathbb E\!\bigg[\sum_{k=0}^{N}\gamma^k
\big(Z^\top(k)Q_ZZ(k)+\tau_s^\top(k)R_\tau\tau_s(k)\big)\bigg],
\end{aligned}
\end{equation}
where $R_\tau=D_\tau^{-1}R_ZD_\tau^{-1}$ and $D_\tau$ is a positive diagonal matrix of reference force and moment scales specified. Fig.~\ref{fig:costs} (b) shows the cumulative expected cost estimated by computing the quadratic penalties for each realization and then averaging, i.e., $\mathcal J^*_{\tau,N}=13542.73$. Fig.~\ref{fig:costs} (b) evaluates the proposed controller under the physical-input criterion.

The source code and complete simulation settings are provided at:

{\footnotesize\url{https://github.com/Yiting-Luo/Python-Code-for-Numerical-Example/raw/refs/heads/main/Accompanying_Python_Code.zip}.}
\begin{figure}[!htbp]
\centering
\subfloat[]{%
\includegraphics[width=0.48\columnwidth]{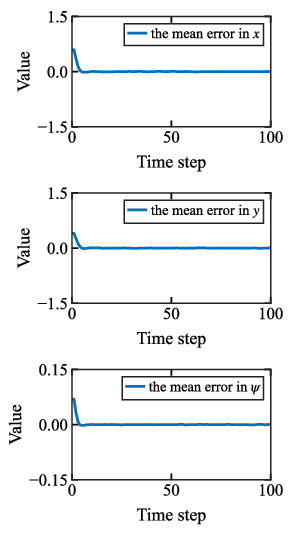}%
\label{2}}\hfill
\subfloat[]{%
\includegraphics[width=0.48\columnwidth]{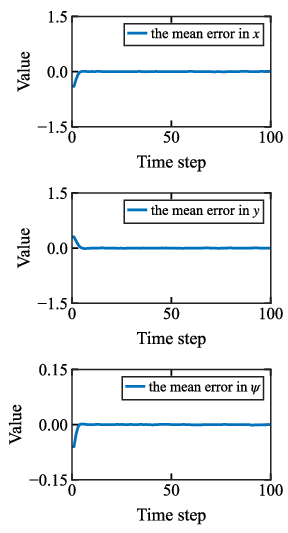}%
\label{4}}
\caption{(a) Simulation curves for the error in $x$, $y$ and $\psi$ of leader AUV.\\ 
(b) Simulation curves for the error in $x$, $y$ and $\psi$ of follower AUV.}
\label{fig:tracking}
\end{figure}

\begin{figure}[!htbp]
\centering
\subfloat[]{%
\includegraphics[width=0.48\columnwidth]{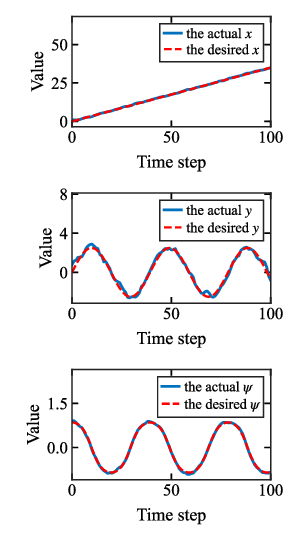}%
\label{1}}\hfill
\subfloat[]{%
\includegraphics[width=0.48\columnwidth]{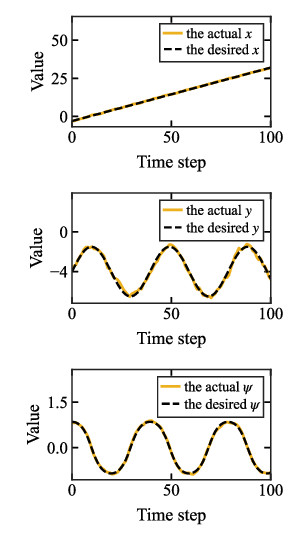}%
\label{3}}
\caption{(a) Simulation curves for the actual and the desired $x$, $y$ and $\psi$ of leader AUV.\\
 (b) Simulation curves for the actual and the desired $x$, $y$ and $\psi$ of follower AUV.}
\label{fig:components}
\end{figure}
\begin{figure}[!htbp]
\centering
\subfloat[]{%
\includegraphics[width=0.48\columnwidth]{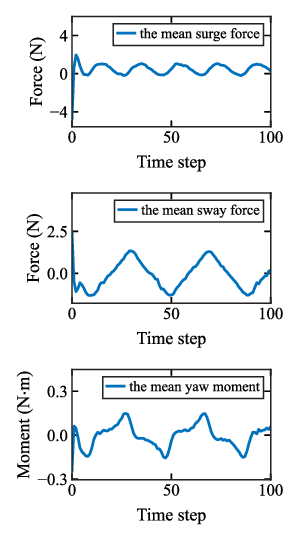}%
\label{fig:forcesLeader}}\hfill
\subfloat[]{%
\includegraphics[width=0.48\columnwidth]{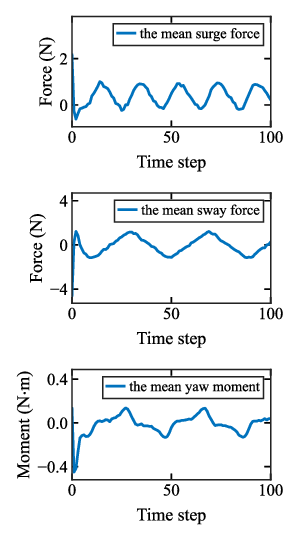}%
\label{fig:forcesFollower}}
\caption{(a) Simulation curves for surge forces, sway forces, and yaw moments of leader AUV.\\
 (b) Simulation curves for surge forces, sway forces, and yaw moments of follower AUV.}
\label{fig:forces}
\end{figure}
\begin{figure}[!htbp]
\centering
\subfloat[]{%
\includegraphics[width=\columnwidth]{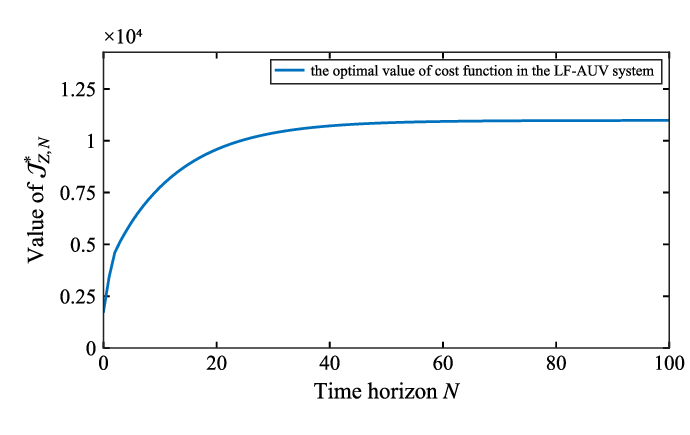}%
\label{5}}\par
\subfloat[]{%
\includegraphics[width=\columnwidth,trim=0 0 0 10bp,clip]{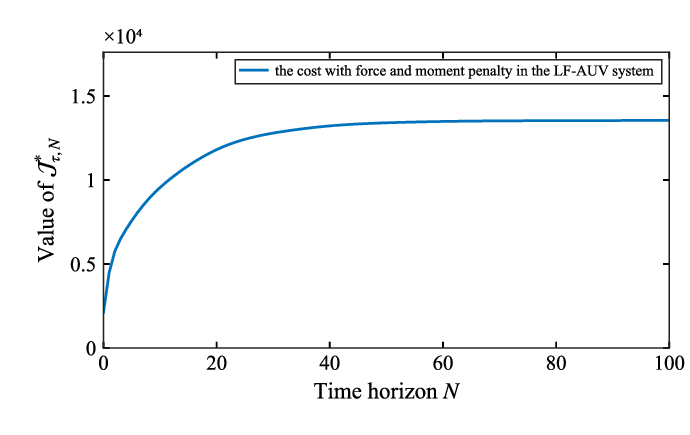}%
\label{fig:physicalCost}}
\caption{(a) Simulation curve for the optimal value of the cost function in the LF-AUV system.\\
 (b) Simulation curve for the optimal value of the cost function with force and moment penalties in the LF-AUV system.}
\label{fig:costs}
\end{figure}

\section{Conclusion}\label{ss8}
In this paper, the optimal decentralized LQ control problem for LFNSs with asymmetric information has been thoroughly studied in both finite-horizon and infinite-horizon. To address the challenges induced by asymmetric information structure, an optimal iterative estimator has been derived and the optimal decentralized control strategy has been developed by decoupling the FBSDEs through orthogonal decomposition. Moreover, the necessary and sufficient conditions for feedback stabilization of infinite-horizon have been proposed. As an application, the theoretical results of this paper are applied to trajectory tracking control of a LF-AUV system, with numerical simulations demonstrating the effectiveness of the proposed method. For future research, we will focus on extending these results to more general decentralized control scenarios. Furthermore, we will implement and test the proposed algorithms through experimental validation on existing AUV platforms.

\vspace{-1cm}
\begin{IEEEbiography}[{\includegraphics[width=1in,height=1.25in,clip,keepaspectratio]{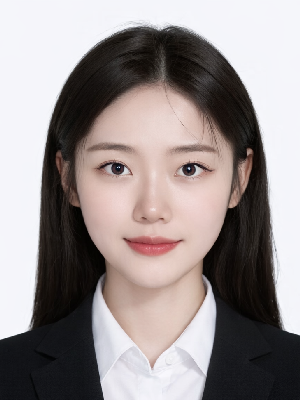}}] 
{Yiting Luo} received the Bachelor's degree in Engineering from Henan Polytechnic University, Jiaozuo, China, in 2023. She is currently pursuing a Ph.D. degree in Control Science and Engineering at the College of Intelligent Systems Science, Harbin Engineering University, Harbin, China. 

Her research interests include optimal control with asymmetric information, control and applications of underwater vehicles.
\end{IEEEbiography}
\vspace{-1cm}
\begin{IEEEbiography}[{\includegraphics[width=1in,height=1.25in,clip,keepaspectratio]{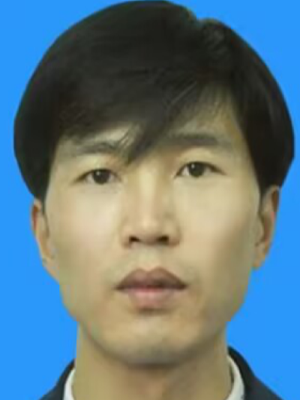}}] 
{Wei Wang} received the Ph.D. degree in Control Science and Engineering from Shenzhen Graduate School, Harbin Institute of Technology, China, in 2010. He is currently a Professor with Shandong University. 

His research interests include optimal control, estimation for delayed systems and distributed control and estimation. 
\end{IEEEbiography}
\vspace{-1cm}
\begin{IEEEbiography}[{\includegraphics[width=1in,height=1.25in,clip,keepaspectratio]{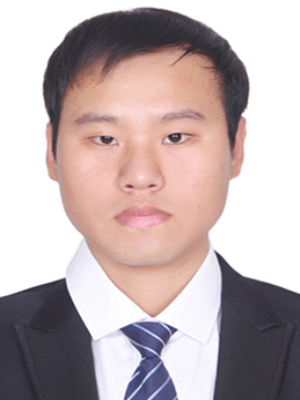}}] 
{Qingyuan Qi} received the B.S. degree in Mathematics and the Ph.D. degree in Control Theory and Engineering from Shandong University, Jinan, China, in 2012 and 2018, respectively. He was an Associate Professor with Qingdao University, Qingdao, China, and a Postdoctoral Fellow with Nanyang Technological University, Singapore, from 2019 to 2020. He is currently an Associate Professor with College of Intelligent Systems Science, Harbin Engineering University, Harbin, China.

His research interests include optimal control and stabilization, stochastic systems and control of underwater vehicles.
\end{IEEEbiography}

\end{document}